\documentclass[10pt]{amsart} 
\usepackage{amsfonts}
\usepackage{amssymb}
\usepackage{amsmath}
\usepackage[arrow,matrix,curve]{xy}

\usepackage{color}

\newcommand{\C}{{\mathbb C}}

\newcommand{\Z}{{\mathbb Z}}

\renewcommand{\H}{{\mathbb H}}

\newtheorem{theorem}{Theorem}[section]
\newtheorem{lemma}[theorem]{Lemma}
\newtheorem{remark}[theorem]{Remark}
\newtheorem{example}[theorem]{Example}
\newtheorem{corollary}[theorem]{Corollary}
\newtheorem{proposition}[theorem]{Proposition}
\newtheorem{definition}[theorem]{Definition}

\def\cal{\mathcal}

\newcommand{\calh}[0]{{\cal H}}

\newcommand{\calf}[0]{{\cal F}}
\newcommand{\call}[0]{{\cal L}}
\newcommand{\calg}[0]{{\cal G}}
\newcommand{\calb}[0]{{\cal B}}
\newcommand{\cala}[0]{{\cal A}}
\newcommand{\cale}[0]{{\cal E}}
\newcommand{\calk}[0]{{\cal K}}

\newcommand{\cals}[0]{{\cal S}}
\newcommand{\calm}[0]{{\cal M}}

\newcommand{\calt}[0]{{\cal T}}

\author[B. Burgstaller]{Bernhard Burgstaller}
\email{bernhardburgstaller@yahoo.de}

\begin{document}

\title[Aspects of equivariant $KK$-theory]{Aspects of equivariant $KK$-theory in its generators and relations picture}

\begin{abstract}

We give a new proof of the universal property
of 
$KK^G$-theory with respect to stability, homotopy invariance and split-exactness for $G$ a locally compact group,
or 
a 
locally compact (not necessarily Hausdorff) groupoid, or a countable inverse semigroup which is relatively short and conceptual.
Morphisms 
in the generators and relations 
picture of
$KK^G$-theory 
are brought to a particular simple form.

\end{abstract}

\maketitle 



\section{Introduction}

In 
\cite{kasparov1981} Kasparov has introduced $KK$-theory for $C^*$-algebras, a bivariant $K$-theory fusing $K$-homology with $K$-theory.
Afterwards, Cuntz found another description of $KK$-theory by interpreting $KK$-theory elements
as quasihomomorphisms and showing that $KK$-theory elements act on 
stable, homotopy invariant, split-exact 
functors 
from the $C^*$-category to abelian groups,
%
see the relevant papers \cite{cuntz1983,cuntz1984,cuntz1987}.
Based on 
this categorial finding, Higson \cite{higson}
proved that every stable, homotopy invariant, split-exact functor from the $C^*$-category
to an additive category 
uniquely `extends' to $KK$-theory. 
%
%
Actually, Kasparov considered the $C^*$-algebras to be $G$-equivariant with respect to a compact group $G$, and generalized this in 
\cite{kasparov1988} to locally compact, second-countable groups.
Cuntz and Higson's findings were done non-equivariantly, and
in \cite{thomsen} 
Thomsen generalized 
Higson's 
result 
to $G$-equivariant $KK$-theory.
%

By Cuntz and Higson's findings
it is 
evident that the category of equivariant $KK$-theory restricted to the category of separable $C^*$-algebras
may be expressed by generators and relations,
where the generators are the $C^*$-homomorphisms and other synthetical inverse morphisms.
We have described this in more details in
\cite{bgenerators} and called the category $GK$
for simplicity.
 %
In \cite{buniversal} we  
have 
shown that $S$-equivariant $KK$-theory for a countable 
discrete inverse semigroup $S$ also satisfies the universal property.
The proof works also almost unchanged for a locally compact (not necessarily Hausdorff)
groupoid, see corollary \ref{corollary22}.

In \cite{bremarks} it was noted that in $GK$-theory morphims may be written as a short product 
$$a e^{-1} \Delta_s b  \Delta_t f^{-1}$$
where $a,b$ are homomorphisms, $e,f$ are corner embeddings and $\Delta_s,\Delta_t$ are synthetical splits of split exact sequences.
%
To show this one takes a morphism in $GK$-theory, interprets it as a Kasparov element in $KK$-theory and goes back to $GK$-theory by the functor constructed in the proof of the universal property of $KK$-theory.

In this paper we explore the category $GK$ further by simplifying the expression of morphisms in $GK$ directly in $GK$.
This is done 
in an equivariant setting 
with respect to a 
locally compact (not necessarily Hausdorff) groupoid or inverse semigroup $G$.
The whole machinery we present is category theoretically very visual and close to ordinary $C^*$-theory.
On the other hand, the picture in $GK$-theory is still very close to the Kasparov picture. 
Actually we solve all the 
harder problems by using $KK$-theory, particularly the Kasparov technical theorem encoded in the Kasparov product.
Together with many ideas taken from $KK$-theory, for example the Kasparov stabilization theorem.

Beside these benefits, we get a new 
proof of the universal property of equivariant $KK$-theory as a byproduct.
Moreover we can 
improve the word length of the above product to
$$a e^{-1} \Delta_s f^{-1}$$

Note that 
we have not optimized the exposition to achieve as fast as possible the main result
theorem \ref{theorem212} and as a corollary the universal property of $KK$-theory, corollary \ref{cor415}.
If one is interested in 
these proofs 
one may 
save several pages. 
Indeed one
 would need of sections \ref{sec5} to \ref{sec8}
only a few lemmas
(lemmas \ref{lemma21}, \ref{lemma51}, \ref{lemma272}, \ref{lemma100}, \ref{lemma215},
\ref{lemma86}, \ref{lemma91}, \ref{lemma123}).
As remarked in section \ref{sec11}, 
its contained proofs could be extremely 
cut short. 
The proof of the universal property actually really begins in section \ref{sec9}.

We give a short overview of the paper.
In section \ref{sec2}
we remark that
$KK^\calg$ for a locally compact (not necessarily Hausdorff) groupoid $\calg$
satisfies the universal property by the same proof as for inverse semigroups.
In section \ref{sec3} we briefly recall $GK$-theory, and in section \ref{sec4} we introduce the functor from $GK$-theory to $KK^G$-theory.
In section \ref{sec5} we define the important concept of double split exact sequences in the equivariant setting.
The idea to use non-equivariant double split exact sequences in $KK$-theory is essentially due to Cuntz \cite{cuntz1984}, 
and the additional matrix trick to handle equivariance 
goes back at least to Connes \cite{connes1973} and was used by Thomsen \cite{thomsen}.
Fundamental is the construction of 
split exact sequences in section \ref{sec6}, 
also 
used by Kasparov 
\cite{kasparov1981}.
In section \ref{sec7} we discuss $G$-actions on 
a 
$2 \times 2$-matrix algebra used in our framework.
In section \ref{sec8} we demonstrate various computations in $GK$-theory, including sideways which are not relevant for the main results.
Actually, in lemmas \ref{lemma272} and \ref{lemma812} we see how Kasparov's definition of the $KK$-groups come out naturally and suggest itself in our framework.
In section \ref{sec9} we introduce the functor from $KK^G$-theory to $GK$-theory, and in section \ref{sec10} we detect the first relations to the functor in the other 
direction.
Section \ref{sec11} shows an important concept by 
Kasparov, and technically simplified by Connes-Skandalis, to prepare a pushout-construction used
in the two 
consecutive sections.
In section \ref{sec12} we use the Kasparov product for the fusion of a synthetical split with a double split exact sequence.
We do not need the Kasparov product any more
in section \ref{sec12}
to fuse analogously a double split exact sequence with the inverse of a corner embedding.
In section \ref{sec13} we show by induction on the length of a word of a morphism in $GK$-theory that it can be simplified to the simple form as stated above.
As a corollary we obtain the universal property of equivariant $KK$-theory.

\section{The universal property of $KK^\calg$ for groupoids $\calg$}			\label{sec2}

Let $G$ be a second countable locally compact group,
a second countable locally compact (not necessarily Hausdorff)
groupoid, or a countable inverse semigroup.
The category of separable $G$-equivariant $C^*$-algebras and their $G$-equivariant homomorphisms
is denoted by $C^*$. 
We often use the term `non-equivariant' when we want to ignore any $G$-action or $G$-equivariance.
All Hilbert modules are assumed to be countably generated,
and all $C^*$-algebras are separable.

The $C^*$-algebra of adjoint-able operators on a Hilbert $B$-module $\cale$ is denoted by $\call_B(\cale)$ or $\call(\cale)$ and its 
two-sided closed ideal of `compact' operators by
$\calk_B(\cale)$. 

The reference for group equivariant $KK$-theory is Kasparov \cite {kasparov1988}, for groupoid equivariant $KK$-theory it is Le Gall \cite{legall},
and for inverse semigroup equivariant $KK$-theory it is \cite{bsemimultikk}, or see \cite{buniversal}
for a 
summary of the definitions.
(We use the slightly adapted `compatible' version of  equivariant $KK$-theory as 
in \cite{buniversal}.)
The category of $G$-equivariant $C^*$-algebras
with the Kasparov groups as morphisms 
is denoted by $KK^G$.

In this paper we write compositions of morphisms in a category and compositions of functions from 
left to right. That is, for instance, if $f:A \rightarrow B$ and $g : B \rightarrow C$ are maps, then we write
$f g$ for $g \circ f$, where composition operator $\circ$ is used in the usual sense from right to left.
This will go as far as that we write $fg(x)$ for $g(f(x))$.
In spaces of operators like $\call(\cale)$ we use the multiplication in the usual sense, that is,
$S T$ means $S \circ T$ for $S,T \in \call(\cale)$, but to avoid confusion, we mostly write
$S \circ T$.

For a $G$-action $S$ on a Hilbert module $\cale$ 
we write ${\rm Ad}(S)$ for the $G$-action $\gamma_g(T) = S_g \circ T \circ S_{g^{-1}}$ 
on $\call(\cale)$.
For a unitary $U \in A$ we write
${\rm Ad}(U)$ for the 
$*$-automorphism 
$f(a)=U a U^*$ on $A$.

If we notate 
grading in a Kasparov element as for example in 
$[\pi_1 \oplus \pi_0,\cale_1 \oplus \cale_0,F] \in KK^G(A,B)$, then the first notated summand $\cale_1$ always means the odd graded part, 
and the second summand $\cale_0$ the 
even graded part.
We also write $[\sigma_1 + \sigma_0,\cale_1 \oplus \cale_0,F]$, where then $\sigma_1(a) = \pi_1(a) \oplus 0$ and $\sigma_0 (a)= 0 \oplus \pi_0(a)$.

A map into multiplication operators 
like the canonical embedding $f:A \rightarrow \call_A(A)$ is often sloppily denoted by ${\rm id}$, or written as $f(a)=a$.
The identity map is often denoted by $1$ (for example in $T \otimes 1$), or by ${\rm id}$.

For a non-equivariant $C^*$-algebra $A$ we write $e_{11},e_{22} : A \rightarrow M_2(A)$ for the two corner embeddings 
into the upper left and lower right corner respectively.

We denote $A \otimes (C_0([0,1]),{\rm triv})$ by $A[0,1]$, where `triv' means trivial $G$-action.

In \cite{buniversal} we have proven the universal property
of $G$-equivariant $KK$-theory when $G$ is a countable discrete inverse semigroup.
In this section we remark that the proof works verbatim also when $G$ is a locally compact,
not necessarily Hausdorff groupoid.

Indeed, let $\calg$ be a locally compact groupoid with base space $X$. 
At first we may consider it as a discrete
inverse semigroup $S$ by adjoining 
a zero element to $\calg$, i.e. 
set $S:= \calg \cup \{0\}$.

A $\calg$-action $\alpha$ on $A$ is then fiber-wise just like an inverse semigroup $S$-action on $A$ (the zero element $0 \in S$ acts always as zero), with the additional property that it is continuous in the sense that
it forms a map $\alpha:s^* A \rightarrow r^* A$.
We cannot, as in inverse semigroup theory, say that $\alpha_{s s^{-1}}( A)$ is a subalgebra of $A$ ($s \in S$),
because this instead we would interpret as a fiber $A_{s s^{-1}}$ of $A$. But all computations done for inverse semigroups would be the same if we did it for a groupoid on fibers. 
That is why we need only 
take care that every introduced $\calg$-action 
is continuous.

But the introduced actions, or similar constructions are just:

\begin{itemize}

\item
Cocycles: The definition \cite[def. 5.1]{buniversal} has to be replaced by the analogous definition \ref{def021} below.

\item
Unitization: One replaces \cite[def. 3.3]{buniversal} by \cite{legall}.

\item
Direct sum, internal, external tensor product:
It is clear that these constructions are also continuous for groupoids.

\item
For an element $[T,\cale] \in KK^{\calg}(A,B)$ one has the condition that the bundle
$g \mapsto g(T_{s(g)}) - T_{r(g)} \in \calk(\cale_{r(g)})$ is in $r^* \calk(\cale)$. Here one has also additionally to check continuity.

\end{itemize}

\begin{definition}				\label{def021}
{\rm
Let $(A,\alpha)$ be a $\calg$-algebra.
Set the $\calg$-action on $\call_A(A) \cong \calm(A)$ to be $\overline \alpha
 := {\rm Ad}(\alpha)$. 
An  $\alpha$-cocycle is a unitary $u$ in
$r^* \big (\calm(A) \big )$ such that
$$u_{gh}= \overline \alpha_g(u_h)$$
in $\calm(A)_{r(g)}$ for all $g,h \in \calg$ with
$s(g)=r(h)$.
}
\end{definition}


In this way it is (almost) clear that the results of \cite{buniversal} hold also in the locally compact (not necessarily Hausdorff) groupoid
equivariant setting.

\begin{corollary}			\label{corollary22}
Let $\calg$ be a locally compact 
(not necessarily Hausdorff) groupoid.
Then $KK^\calg$ is the universal stable, homotopy invariant and split exact category deduced from 
the 
category of $\calg$-equivariant, separable, ungraded $C^*$-algebras.
\end{corollary}

From now on,
if nothing else is stated, we assume that $G$ is an inverse semigroup.
This is almost invisible, except at least in
corollary \ref{lemma272}.(iv) and it is obvious how to adapt it to groupoids 
keeping the above remarks in mind.

\section{$GK$-theory}

\label{sec3}


We are going to recall the definition of $GK$-theory
(``Generators and relations $KK$-theory", the group $G$ is not indicated, instead we may also write $GK^G$) 
for which we refer for more details to 
\cite{bgenerators}.
The split exactness axiom is slightly but equivalently altered,
see \cite[Lemma 3.7]{bremarks}. 


\begin{definition}		\label{def121}
{\rm

Let $GK$ be the following 
category.
Object class of $GK$ is the class of all separable $G$-algebras.

Generator 
morphism class is the collection of all $G$-equivariant $*$-homomorphisms $f:A \rightarrow B$ (with obvious source and range objects) and the collection of the following ``synthetical" morphisms:

\begin{itemize}

\item
For every equivariant corner embedding $e:(A,\alpha) \rightarrow (A \otimes \calk,\delta)$
($\delta$ need not be diagonal but can be any $G$-action) add a morphism called $e^{-1}: (A \otimes \calk,\delta) \rightarrow (A,\alpha)$.

\item
For every equivariant short split exact sequence
\begin{equation}			\label{splitexact}
\cals: \xymatrix{0 \ar[r] & (B,\beta)  \ar[r]^j & (M,\delta)  \ar[r]^f & (A,\alpha) \ar[r] \ar@<.5ex>[l]^s & 0}
\end{equation}
in $C^*$ add a morphism called
$\Delta_{\cals}: (M,\delta) \rightarrow (B,\beta)$
or $\Delta_s$ if $\cals$ is understood.

\end{itemize}

Form the free category of the above generators
together with free addition and substraction of morphims having same range and source 
(formally this is like the free ring generated by these generator morphisms, but one can only add 
and multiply if 
source and range fit together)
and divide out the following relations
to turn it into the category $GK$:

\begin{itemize}

\item
({\em $C^*$-category})
Set 
$g \circ f = f g$ for all $f \in C^*(A,B)$
and $g \in C^*(B,C)$.

\item
{(\em Unit})
For every object $A$,
${\rm id}_A$ is the unit morphism.

\item
({\em Additive category})
For all 
diagrams
$$\xymatrix{A  \ar[r]^{i_A} &  A \oplus B  \ar[r]^{p_B} 
\ar@<.5ex>[l]^{p_A}
&  A  \ar@<.5ex>[l]^{i_B} }$$
(canonical projections and injections) set
$1_{A \oplus B}= p_A i_A + p_B i_A$.

\item
({\em Homotopy invariance})
For all homotopies $f:A \rightarrow B[0,1]$ in $C^*(A,B)$ set $f_0 = f_1$.

\item
({\em Stability})
All corner embeddings $e$ are invertible with inverse $e^{-1}$.

\item
({\em Split exactness})
For all split exact sequences (\ref{splitexact}) set
\begin{eqnarray*}
&& 1_B   =  j \Delta_s  \\
&& 1_M   =    \Delta_s j + f s
\end{eqnarray*}

\end{itemize}

}
\end{definition}

\begin{remark}			\label{rem12}
{\rm
(i) If we have given a split exact sequence (\ref{splitexact}) then it splits completely as linear maps, that is, $j$ has a linear split
$t:M \rightarrow B$ with $t(x)=j^{-1}(x- fs(x))$,
and the split exactness relations of definition \ref{def121} are satisfied for such a split $\Delta_s:=t$.
Thus $\Delta_s$ 
may be viewed as a substitute for this linear split $j$, and it is often useful to think about 
$\Delta_s$ as $j$ in heuristical considerations.

(ii)
Set-theoretically $M= j(B) + s(A)$ in (\ref{splitexact}), and this is a direct linear sum by the last point.

(iii)
If (\ref{splitexact}) has the flaw that it is not exact in the middle but only $j(B) \subseteq {\rm ker}(f)$
then this can be repaired by restricting $M$ to the $G$-subalgebra $N:= j(B)+ s(A)$.

(iv)
$j,s,f$ in (\ref{splitexact}) all influence $\Delta_s$. This is clear for the linear split $t$, and so this must be even more true for the free generator $\Delta_s$.

(v) If we have given an additional homomorphism $u:A \rightarrow M$ in (\ref{splitexact}) then this is a second split for $f$ if and only if
$$u(a)-s(a) \in j(B) \quad \forall a \in A$$

(vi) Given (\ref{splitexact}), we have $s \Delta_s =0$
because
$s \Delta_s = s \Delta_s j \Delta_s =
s (1-fs) \Delta_s =0$.

(vii) If $f: (A,\alpha) \rightarrow M_n(A,\delta)$
is a corner embedding, then it is invertible in $GK$. In fact, $g :(M_n(A),\delta) \rightarrow (M_n(A) \otimes \calk, \delta \otimes {\rm triv})$
is an invertible corner embedding, as well as $fg$, so $f$ itself must be invertible.

}

\end{remark}

\section{The functor $\cala$}

\label{sec4}

Since equivariant $KK$-theory is stable, homotopy invariant and split exact, there is a functor from the 
univerally defined $GK$-theory to $KK$-theory.
It can be concretely constructed as follows, see \cite{thomsen} 
in the group equivariant case, 
or \cite[section 4]{buniversal} for
the inverse semigroups equivariant setting.

\begin{definition}		\label{def21}
{\rm
Define $\cala:GK \rightarrow KK^G$ to be the additive functor which is identical on objects
and as follows on generator morphisms:

(i)
For an equivariant $*$-homomorphism
$f: A \rightarrow B$ we put
$$\cala(f) = f_*([{\rm id},A,0]) \in KK^G(A,B)$$
 
(ii)
For a corner embedding
$e:(A,\alpha) \rightarrow (A \otimes \calk,\delta)$ in $C^*$ we set
$$\cala(e^{-1}) = 
[{\rm id}, \big( (A \otimes \calk) E, \delta\big), 0] 
\in KK^G(A \otimes \calk,A)$$
where $E:= e(A) \subseteq A \otimes \calk$ is the $G$-invariant corner $G$-algebra.

(iii)
For a split exact sequence (\ref{splitexact})
we define
the equivariant $*$-homomorphism
\begin{equation}		\label{chi}
\chi:M \rightarrow \call_B(B): \chi(x)(b) = j^{-1}(x j(b))
\end{equation}
and 
$$\cala(\Delta_s)=
[fs \chi \oplus \chi, (B \oplus B, \beta \oplus \beta), F] \in KK^G(X,A)$$
where $B \oplus B$ has the grading $- \oplus +$ and $F$ is the flip operator.

}
\end{definition}

\section{Double split exact sequences}

\label{sec5}

Throughout, $(A,\alpha)$ and $(B,\beta)$ are $G$-algebras.

\begin{definition}			\label{def31}
{\rm
A {\em double split exact sequence}
is a diagram of the form
$$\xymatrix{
0 \ar[r] &  
(B,\beta)  \ar[r]^j & (M,\gamma) \ar[rrr]^f 
\ar[dr]^{e_{11}}
&
& &  (A,\alpha)  \ar@<.5ex>[lll]^{s}
\ar[ld]^{t} 		\ar[r]
& 0\\
& 
&
& (M_2(M), \theta) 
& (M,\delta) \ar[l]^{e_{22}}		
}
$$
where all morphisms in the diagram are equivariant $*$-homomorphisms,
the first line is a split exact sequence in $C^*$, $t$ is another split in the sense that
$t f = 1_A$ in non-equivariant $C^*$, and $e_{ii}$ are the 
corner embeddings.


}
\end{definition}

\begin{definition}			\label{def32}
{\rm
Consider a double split exact sequence as above.
We denote by $\mu_\theta$, or $\mu$ if $\theta$ is understood, the morphism
$$\mu: (M,\delta) \rightarrow (M,\gamma) :
\mu = e_{22} e_{11}^{-1}$$
in $GK$.
The {\em morphism in $GK$ associated to the 
double split exact sequence} is 
$$t \mu \Delta_{s}$$
}
\end{definition}

We use sloppy language and say for example ``the diagram is $t \mu \Delta_s$ in $GK$'', or two double split exact sequences are said to be ``equivalent" 
if their associated morphisms are.
Throughout, the short notation for the above double split exact sequence will be
$$\xymatrix{
(B,\beta)  \ar[r]^j & (M,\gamma) \ar[r]^f 
&  (A,\alpha)  \ar@<.5ex>[l]^{s,t}
} $$
Notating such a diagram, it is implicitely understood that this is a double split exact sequence as above if nothing else is said.
Often $s,t$ is stated as $s_\pm$, 
which has to be read as $s_-,s_+$.
The $G$-action $\theta$ of definition \ref{def31}
will sometimes be called the ``$M_2$-action of the double split exact sequence" 
for simplicity.

\begin{example}
{\rm
Assume $G$ is the trivial group.
%
Then $\mu=1$ is the identity in $GK$ because $e_{11}$ and $e_{22}$ are homotopic by a rotation in $C^*$.

Consequently, we have double split exact sequences in the more usual sense and
$t \mu \Delta_s = t \Delta_s$.
Moreover,
$t \Delta_s = (t-s)\Delta_s $.

}
\end{example}


\begin{lemma}			\label{lemma21}
{\rm
Consider two double split exact sequences
which are connected by three morphisms
$b,\Phi,a$ in $GK$ as in this diagram:
$$\xymatrix{
B  \ar@<.5ex>[r]^i
\ar[ddd]^b
 & M \ar[rrr]^f 
\ar[rd]^{e_{11}}
\ar@{..>}[ddd]^\phi
 &
& & A  \ar@<.5ex>[lll]^{s_-}
\ar[ld]^{s_+} 
\ar[ddd]^a
\\
&& M_2(M) 	\ar[d]^\Phi 
& M \ar[l]^{e_{22}}		
\ar@{..>}[d]^\psi   \\
&& M_2(N)  
& N \ar[l]^{f_{22}}		\\
D  \ar@<.5ex>[r]^j & N  \ar[rrr]^g 
\ar[ur]^{f_{11}}
 &
& & C  \ar@<.5ex>[lll]^{t_-}
\ar[ul]^{t_+}
} 
$$

Here we have defined
$$\phi:= e_{11} \Phi f_{11}^{-1}
\qquad
\psi:= e_{22} \Phi f_{22}^{-1}
$$

(i)
Then
for the commutativity of the left rectangle of the diagram we note
$$\Delta_{s_-} b = \phi \Delta_{t_-}
\qquad \Leftarrow   
\qquad
i \phi = b j \quad {\rm and}
\quad 
f s_- \phi = \phi g t_-$$

(ii)
For 
commutativity within the right big square of the diagram we observe
$$s_+ \mu \phi = a t_+ \mu 
\qquad \Leftrightarrow   
\qquad
s_+ \psi = a t_+
$$

(iii)
Consequently, commutativity of double split exact sequences in this diagram can be decided as 
$$s_+ \mu \Delta_{s_-} b
= a t_+ \mu \Delta_{t_-} 
\qquad \Leftarrow   
\qquad
\mbox{Conditions of (i) and (ii) hold true}
$$

}
\end{lemma}

\begin{proof}
(i) 
We compute
$$\phi \Delta_{t_-} = (\Delta_{s_-} i + f s_-) \phi \Delta_{t_-}
= \Delta_{s_-} b j \Delta_{t_-} + \phi g t_- \Delta_{t_-} 
= \Delta_{s_-} b$$

(ii) This is clear by commutativity of involved rectangles in the diagram and invertibility of all corner embeddings. (iii) Also clear.
\end{proof}

We will exclusively 
encounter this situation:

\begin{remark} 
\label{cor22}
{\rm
Assume that $\Phi= \phi \otimes 1_{M_2}$
is equivariant
for a non-equivariant 
$*$-homomorphism $\phi:M \rightarrow M$.

Then this $\phi$ is the $\phi$ and $\psi$ in the above diagram as non-equivariant maps, and both
are automatically equivariant as maps as 
entered in the diagram.
%
}
\end{remark}


\section{The $M \square A$-construction}
%
%

\label{sec6}

We shall use the following standard procedure
to produce split exact sequences,
and this is in fact 
key:



\begin{definition}		\label{lemma35}
{\rm
Let $i: (B,\beta) \rightarrow (M,\gamma)$ be an equivariant injective $*$-homomorphism such that
the image of $i$ is an ideal in $M$.
Let $s: (A,\alpha) \rightarrow (M,\gamma)$ be an equivariant
$*$-homomorphism.
Then we 
define the equivariant 
$G$-subalgebra
$$M \square A := 
\{ (s(a)+i(b),  a) \in M \oplus A|\,a \in A, b \in B\} $$ 
of $(M \oplus A, \gamma \oplus \alpha)$.
The $G$-action on $M \square A$ is denoted by
$\gamma \square \alpha$.
In particular we have a split exact sequence
$$\xymatrix{0 \ar[r] & B  \ar[r]^j & M \square A \ar[r]^f & A \ar[r] \ar@<.5ex>[l]^{s \square 1} & 0},$$
where $j(b)= (i(b),0)$, $f(m,a)=a$ and $(s \square 1)(a) :=(s(a),a)$ for all $a \in A, b \in B, m \in M$.

}
\end{definition}

If we have given a double split exact sequence 
as in definition \ref{def31} with $M$ of the form $M \square A$
then 
it is understood that $j,f$ and $s \square 1$ 
are always of the form as in the last definition.
Moreover, the construction of $M \square A$ refers always to the first notated split $s$, or the split indexed by minus (e.g. $s_- \square 1$)
if it appears in a double split exact sequence.
We also write $s \square$ instead of $s \square 1$ in diagrams.
We denote elements of $M \square A$ by $m \square a :=(m,a)$. The operator $\square$ binds weakly, that is for example, $m+n \square a =(m+n)\square a$.

Non-equivariantly we have 
\begin{equation}				\label{isq}
M_2(M \square A) \cong
M_2(M) \square M_2(A) \subseteq M_2(M) \oplus M_2(A)
\end{equation}
with respect to
$i \otimes 1$ and $s \otimes 1$.

If we have $G$-algebras $(M_2(M),\gamma)$ and $(M_2(A),\delta)$ 
and $(M_2(M \square A), \theta)$
is canonically a $G$-invariant $G$-subalgebra of 
$(M_2(M) \oplus M_2(A),\gamma \oplus \delta)$ 
then we call $\theta$ also $\gamma \square \delta$.

\begin{lemma}
Consider definition \ref{lemma35}.
If we have $G$-algebras $(M_2(M),\gamma)$ and $(M_2(A),\delta)$ and
$M_2(i(B))$ is 
$G$-invariant under $\gamma$
then
$M_2(M \square A)$ is $G$-invariant under the $G$-action 
$\gamma \oplus \delta$ 
if and only if
$$\gamma_g((s\otimes 1)(a))- (s \otimes 1)(\delta_g(a)) \in M_2(i(B))
\quad \forall a \in M_2(A)$$ 
\end{lemma}

\begin{proof}
We apply the isomorphism 
(\ref{isq}) and may work with $i \otimes 1$ and $s \otimes 1$ as in definition
\ref{lemma35}. The proof is then straightworward, or confer the similar proof of
lemma \ref{lemma272}.
\end{proof}

One may observe that 
the non-equivariant splits of the 
exact sequence of 
definition \ref{lemma35}
are exactly the maps of the from $t \square 1$.
We may bring any 
double split exact sequence to this form:

\begin{lemma}		\label{lemma41}
{\rm (i)}
Any double split 
exact sequence as in 
the 
first line of 
this diagram
can be completed to this diagram
$$\xymatrix{
B  \ar[r]^i \ar[d]
& M  \ar[r]^f 
\ar[d]^\phi
& A 
\ar@<.5ex>[l]^{s_\pm} 	
\ar[d]
\\
B  \ar[r]^j  & M \square A \ar[r]^g  & A 
\ar@<.5ex>[l]^{t_\pm}  
}
$$
such that
the first line ist the second line in GK,
that is, $s_+ \mu \Delta_{s_-}
= t_+ \mu \Delta_{t_-}$.

{\rm (ii)}
If the $G$-action on $M_2(M)$ is $\theta$,
then the
$G$-action on $M_2(M \square A)$
is 
$\theta \square \delta$
if and only if the $G$-action $\theta \square \delta$ exists
if and only if $f \otimes 1 : 
(M_2(M),\theta) \rightarrow (M_2(A),\delta)$
is equivariant.

{\rm (iii)}
The
$G$-action on $M_2(M \square A)$
is of the form 
$\theta \square \delta$.

\end{lemma}

\begin{proof}
(i)
Define the second line of the diagram as in definition \ref{lemma35}, that is,
put $g(m \square a)= a$,
$j(b) = i(b) \square 0$ and $t_\pm(a)= s_\pm(a) \square a$.
Define 
$$\phi(m) = m \square f(m)$$

We are going to apply lemma \ref{lemma21}
for $B=D$, $A=C$, $a=1$, $b=1$,
and
$\Phi= \phi \otimes 1_{M_2}$.
Note that $\phi$ is bijective. We define the $G$-action on $M_2(M \square A)$ in such a way that $\Phi$ becomes equivariant.
%
%
By 
remark \ref{cor22}, $\phi=\psi$ in the diagram of lemma \ref{lemma21} and both maps are $G$-equivariant.
We have 
\begin{eqnarray*}
&&i \phi(b) = i(b) \square 0 = j(b)  \\
&&f s_- \phi(m) = f s_-(m) \square f(m) = \phi g (s_- \square 1)(m),
\end{eqnarray*}
which is the condition of
lemma \ref{lemma21}.(i).
Further 
$s_+ \phi(a) = s_+(a) \square a = 
t_+(a)$ 
yields the condition of lemma \ref{lemma21}.(ii).
Hence the claim follows 
from lemma \ref{lemma21}.(iii).
%
%

(ii)
We assume that $(M_2(A),\delta)$ exists and want to see when $\theta \square \delta$ is valid:

Note that 
$m 
\oplus a \in M \square A$ if and only if  $f(m) = a$.
Hence
$\theta_g(m) 
\oplus \delta_g(a) \in M_2(M \square A)$ if and only if $(f \otimes 1)(\theta_g(m))=\delta_g(a)$.
Set $a=(f\otimes 1)(m)$.

(iii)
This follows from (ii) and corollary \ref{cor79},
which is independent from this lemma.
\end{proof}

\section{Actions on $M_2(A)$}

\label{sec7}

In this section we want 
to inspect closer how
a $M_2$-action of a double split exact sequences looks like.
%
%
This is a key lemma:

\begin{lemma}			\label{lemma51}
Let $S,T$ be two $G$-actions on a Hilbert $(B,\beta)$-module $\cale$.
Then
$$
\alpha_g
\left (\begin{matrix} x & y\\
 z & w
\end{matrix}
\right ) =
\left (\begin{matrix} S_g x S_{g^{-1}} & S_g y T_{g^{-1}}\\
 T_g z S_{g^{-1}} & T_g w T_{g^{-1}}
\end{matrix} \right ) 
$$
defines a $G$-action $\alpha$ on $M_2\big ( \call_B(\cale) \big)$

This is actually the inner 
action ${\rm Ad} (S \oplus T)$ 
on $\call_B \big ((\cale,S) \oplus (\cale,T) \big )$.

\end{lemma}

\begin{definition}
{\rm
The
$\alpha$ of the last lemma is also denoted by
${\rm Ad}(S \oplus T)$ or
${\rm Ad}(S,T)$.
}
\end{definition}


\begin{lemma}		\label{lemma61}
Let $(A,\alpha)$ and $(A,\delta)$ be $G$-$C^*$-algebras.

Let $(M_2(A),\theta)$ be a $G$-algebra 
and the corner embeddings
$e_{11} :(A,\alpha) \rightarrow
(M_2(A),\theta)$ and
$e_{22} :(A,\delta) \rightarrow
(M_2(A),\theta)$ be equivariant.



Then
$\theta$ is of the form

$$
\theta_g
\left (\begin{matrix} x & y\\
 z & w
\end{matrix}
\right ) =
\left (\begin{matrix} \alpha_g(x) & \beta_g(y)\\
 \gamma_g(z) & \delta_g(w)
\end{matrix} \right ) 
$$

Also:



{\rm (i)}
One has the relations
\begin{eqnarray}
&&\gamma_g(ax)= \delta_g(a)\gamma_g(x)	
\qquad \gamma_g(xb)= \gamma_g(x) \alpha_g(b)	
	\label{l2}  \\
&& \beta_g(ax)= \alpha_g(a) \beta_g(x)  
\qquad
\beta_g(xb)= \beta_g(x) \delta_g(b)	\\
&&\alpha_g(xy)= \beta_g(x) \gamma_g(y)	
\qquad
\delta_g(xy)= \gamma_g(x) \beta_g(y)	
	\label{l4}  \\
&&\beta_g(y) = {\gamma_g(y^*)}^*	
\qquad \gamma_{gh}= \gamma_g \gamma_h
\qquad \beta_{gh}= \beta_g \beta_h 
	\label{l5}  
\end{eqnarray}

{\rm (ii)}
$(A,\gamma)$ is an imprimitivity Hilbert $((A,\delta),(A,\alpha))$-bimodule,
where the bimodule structure is 
multiplication in $A$,
and the
right inner product is $\langle a,b \rangle=a^*b$ and the left one
is $\langle a,b \rangle =a b^*$.

{\rm (iii)}
Analogously,
$(A,\beta)$ is an imprimitivity Hilbert $((A,\alpha),(A,\delta))$-bimodule.

{\rm (iv)}
Let
$\chi : A \rightarrow 
\call_A(A)$ be the natural embedding.

Then $\alpha$ and $\gamma$ are $G$-actions
on the Hilbert $(A,\alpha)$-module $A$.

Consequently we have the $G$-action
${\rm Ad}(\alpha \oplus \gamma)$
on the matrix algebra $M_2(\call_{(A,\alpha)}(A))$.

The map
$$\chi \otimes 1_{M_2}:(M_2(A), \theta) 
\rightarrow 
(M_2(\call_{(A,\alpha)}(A)), {\rm Ad}(\alpha \oplus \gamma))$$
is a $G$-equivariant injective $*$-homomorphism.

{\rm (v)}
$\gamma$ determines $\theta$ uniquely and completely.

{\rm (vi)}
$\theta$ determines $\alpha, \beta, \gamma$ and $\delta$ uniquely.

{\rm (vii)}
In general, $\alpha$ and $\delta$ do not determine $\gamma$ and thus not $\theta$.

{\rm (viii)}
If we drop all assumptions then we may add:

A $G$-algebra $(A,\alpha)$ and a Hilbert $(A,\alpha)$-module action $\gamma$ on $A$ 
alone ensure the existence of the above $\theta$
with all assumptions and assertions of this lemma.
%


\end{lemma}

\begin{proof}
If $a \in A$, $(a_i)\subseteq A$ an approximate unit of $A$, and we apply $\theta_g$ to 
$$
\left (\begin{matrix} a_i & 0\\
 0 & 0
\end{matrix}
\right )
\left (\begin{matrix} 0 & a\\
 0 & 0
\end{matrix} \right ) 
\left (\begin{matrix} 0 & 0\\
 0 & a_i
\end{matrix}
\right )
$$
then we see that the $\theta_g$ applied to the middle matrix has again the form of the middle matrix.

(i) One computes 
expressions like
$$
\left (\begin{matrix} 0 & 0\\
 \gamma_g(z) & 0
\end{matrix}
\right )
\left (\begin{matrix} \alpha_g(x) & 0\\
 0 & 0
\end{matrix} \right ) 
\quad
\left (\begin{matrix} 0 & 0\\
 0 & \delta_g(x)
\end{matrix}
\right )
\left (\begin{matrix} 0 & 0\\
 \gamma_g(z) & 0
\end{matrix} \right ) 
\quad 
\theta \left (\begin{matrix} 0 & 0\\
 z & 0
\end{matrix} \right )^* 
$$
and uses the fact that $\theta$ is a $G$-action on a $C^*$-algebra.

(ii) Put the first relation of (\ref{l5})
into line (\ref{l4}) 
and then use lines
(\ref{l2}) and (\ref{l4}).

(iv) By (ii), $\alpha$ and $\gamma$ are $G$-actions as claimed, so that the existence of ${\rm Ad}(\alpha \oplus \gamma)$ is by lemma \ref{lemma51}.
By relations (i) one can deduce
\begin{equation}			\label{l7}
\left (\begin{matrix} \alpha_g(x) x' & \beta_g(y)y' \\
 \gamma_g(z)z' & \delta_g(w)w'
\end{matrix} \right ) 
= 
\left (\begin{matrix} \alpha_g(x \alpha_{g^{-1}} 
(x'))  & \alpha_g(y \gamma_{g^{-1}}(y')) \\
 \gamma_g(z \alpha_{g^{-1}}(z')) & \gamma_g(w
\gamma_{g^{-1}}(w'))
\end{matrix} \right ) 
\end{equation}
for all $x,...,w' \in A$, which shows $G$-equivarinace of $\chi \otimes 1$.
In fact, the second matrix line follows directly from (\ref{l2}), and the upper right corner from the first relation of (\ref{l4}).

(v) $\alpha,\delta$ and $\beta$ are determined by $\gamma$ by (\ref{l4}) and the first relation of (\ref{l5}).
 
(viii) By (iv), 
we can 
construct 
${\rm Ad}(\alpha \oplus \gamma)$ and aim to 
define $\theta$ by its restriction. 
To show that the image of $\chi \otimes 1$ is $G$-invariant, we consider the right hand side of (\ref{l7})
and want to construct identity with the left hand side.
For the first column this is clear. 
Setting $\beta$ as in the first identity of (\ref{l5}) we get the upper right corner.
The lower right corner follows from
$$\gamma_g(a a^* \gamma_{g^{-1}} (x)) = \gamma_g(a) \alpha_g(a^* \gamma_{g^{-1}}(x)) = \gamma_g(a) \gamma_g(a)^* x$$

(vii)
Take for example $G=\Z/2$, $A=\C$ (or any $A$), $\alpha=\delta$ the trivial action. Then $\gamma_g(x) = x$ and $\gamma_g(x) = (-1)^g x$ are two valid choices.
%
\end{proof}

\begin{corollary}			\label{cor74}
Let $(M_2(A_i),\theta_i)$ be two $G$-algebras as in lemma \ref{lemma61} ($i=1,2$). Let $\phi:A_1 \rightarrow A_2$ be a non-equivariant $*$-homomorphism.
Then $\phi \otimes 1_{M_2}$ is $G$-equivariant
if and only if it is $G$-equivariant on the lower left corner space 
 $(A_1,\gamma_1)$.
\end{corollary}

\begin{proof}
By lemma \ref{lemma61}.(i), confer also (v).
\end{proof}

\begin{corollary}			\label{cor79}
Consider the double split exact sequence of definition \ref{def31}.

{\rm (i)} Then 
the ideal 
$M_2(j(B))$
is invariant under the action $\theta$.

{\rm (ii)}
The map $f \otimes 1: (M_2(M),\theta)
\rightarrow (M_2(A),\delta)$
is equivariant for the quotient $G$-action 
$\delta$ on
$M_2(A) \cong M_2(M) / M_2(j(B))$.
%
\end{corollary}

\begin{proof}
Write $\theta$ as in lemma \ref{lemma61}.
The corner action $\alpha$ leaves $j(B)$ invariant. 
Hence, by the formulas of lemma \ref{lemma61}.(i) we see that $\gamma, \beta,\delta$ leave an element $b^2 \in j(B)$ invariant.
\end{proof}

\begin{definition}
{\rm
Let $U \in A$ be a unitary in a $C^*$-algebra $A$.
Then we define the $*$-homomorphism
$$\kappa_U:M_2(A) \rightarrow M_2(A) :
\kappa_U
\left (\begin{matrix} x & y\\
 z & w
\end{matrix}
\right ) =
\left (\begin{matrix} x  &  y U^*\\
 U z  & U w U^*
\end{matrix} \right ) 
$$
}
\end{definition}

In other words, $\kappa_U={\rm Ad} (1 \oplus U)$ on $\call_D(D \oplus D) \cong M_2(\call_D(D))$
for $A= \call_D(D)$.

Notice that $\kappa_U^{-1} = \kappa_{U^*}$.

\begin{definition}
{\rm
Let $U \in A$ be a unitary in a $G$-algebra $(A,\alpha)$.
Then we write $\theta^U$ for the $G$-action on $M_2(A)$ defined by
$$\theta^U_g = \kappa_{U} \circ ( \alpha_g \otimes 1_{M_2}) \circ \kappa_U^{-1}$$ 
}
\end{definition}


\begin{lemma}	\label{lemma272}
Let $S,T$ be two $G$-actions on a Hilbert $(B,\beta)$-module $\cale$.

Consider a diagram 
$$\xymatrix{
\calk_B(\cale)  
\ar[r] 
& 
\call_B(\cale) \square A \ar[r]    & A 
\ar@<.5ex>[l]^{s_\pm \square} 
}$$
which is double split exact except that
we have not found a $M_2$-action yet.
But we know that $s_-$ is equivariant with respect to 
${\rm Ad}(S)$ on $\call(\cale)$, and $s_+$ is equivariant with respect to
${\rm Ad}(T)$ on $\call(\cale)$.

Equip $M_2(\call_B(\cale) \oplus A)
\cong \call_B(\cale \oplus \cale) \oplus M_2(A)$ with the 
$G$-action 
$${\rm Ad}(S \oplus T) \oplus (\alpha \otimes 1_{M_2}). $$

Then the following assertions are equivalent:
\begin{itemize}

\item[(i)]
 $s_-(a) \big( S_g T_{g^{-1}} - S_g S_{g^{-1}} \big) \in  \calk_B(\cale)$ for all $g \in G$, $a \in A$


\item[(ii)]
$S_g s_-(a) T_{g^{-1}} - s_-(\alpha_g(a)) \in \calk_B(\cale)$
 for all $g \in G$, $a \in A$

\item[(iii)]
$M_2(\call_B(\cale) \square A)$
is a $G$-invariant subalgebra.

\end{itemize}

In case that there is a unitary $U \in \call(\cale)$ such that $T_g \circ U = U \circ S_g$ for all $g \in G$,
that is if 
$${\rm Ad}(S \oplus T) = \theta^{U}$$
for the $G$-action ${\rm Ad}(S)$ on $\call(\cale)$,
these conditions are
also equivalent to

\begin{itemize}

\item[(iv)]
$s_- (a) \big (g(U) - U g(1) \big ) \in \calk_B(\cale)$ for all $g \in G$, $a \in A$
($G$-action is ${\rm Ad}(S)$).

\end{itemize}

\end{lemma}

\begin{proof}
(ii) $\Rightarrow$ (iii):
Let $x:=s_-(a) + k \square a \in X:= \call(B) \square A$ for $a \in A, k \in \calk(B)$.
Put $x$ into the lower left corner of $M_2(X)$
and apply the $G$-action and see what comes out:
\begin{equation}  	\label{eq21}
S_g \big (s_-(a)+ k \big) T_{g^{-1}} \square \alpha_g(a)
= s_-(\alpha_g(a)) + S_g k   T_{g^{-1}}
\square \alpha_g(a)
\in \call(\cale) \square A
\end{equation}

Similarly we get it for the upper right corner by taking the adjoint in (ii).
For the lower right corner we observe 
$$T_g (s_-(a) + k) T_{g^{-1}} \square \alpha_g(a)
= T_g \big (s_+(a) + k' + k \big) T_{g^{-1}} \square \alpha_g(a)
\in \call(\cale) \square A$$
for a certain $k' \in \calk(B)$ by
remark \ref{rem12}.

(iii) $\Rightarrow$ (ii): By (\ref{eq21})
for $k=0$.
(i) $\Rightarrow$ (ii):
$$S_g s_-(a) T_{g^{-1}} = 
S_g s_-(a) S_g S_{g^{-1}} T_{g^{-1}} 
\equiv s_-(\alpha_g(a)) \mod \calk(\cale)$$
Since $T= U \circ S \circ U^*$, (i) $\Leftrightarrow$ (iv) is obvious.
%
%
%
\end{proof}

By using corollary 
\ref{cor79},
the equivalence between (ii) and (iii)
of the last lemma
may be analogously generalized to diagrams of the form
$\xymatrix{
B  
\ar[r] 
& 
M \square A \ar[r]    & A 
\ar@<.5ex>[l]^{s_\pm \square} 
}$.

\section{Computations with double split exact sequences}

\label{sec8}

From now on, if nothing else is said, the $M_2$-action on $M \square A$ is always understood to be of the form
$\gamma \square (\alpha \otimes 1)$ for $G$-algebras
$(M_2(M),\gamma)$ and 
$(A ,\alpha)$.

Actions on $\call_B(\cale)$ 
will always be of the form ${\rm Ad}(S)$ for a $G$-action $S$ on $\cale$.

\begin{lemma}		\label{lemma100}
Given an equivariant $*$-homomorphism
$f:A \rightarrow B$ 
we get a double split exact sequence
$$\xymatrix{B \ar[r]^i & B \oplus A \ar[r]^g 
 & A  \ar@<.5ex>[l]^{s_\pm} }$$
with $s_-(a)=(0,a), s_+(a)=(f(a),a)$
and one has $f = s_+ \mu \Delta_{s_-}$ in $GK$. 
\end{lemma}

\begin{proof}
The $M_2$-space is $(M_2(A \oplus B),(\beta \oplus \alpha) \otimes 1_{M_2})$, $\mu=1$ by a rotation homotopy, $i(b)=(b,0),g(b,a)=a$ and $\Delta_{s_-} = 
(1- g s_-) i^{-1}$
is just the linear split, see remark \ref{rem12}.
\end{proof}

\begin{lemma}		\label{lemma215}

Given the first line and an
equivariant $*$-homomorphism $\varphi$ 
as in this diagram it can be completed to this diagram
$$\xymatrix{
B  
\ar[r]^{i \square 0} 
& M 
\square A 
\ar[r]    & A 
\ar@<.5ex>[l]^{s_\pm \square 1} 
\\
B \ar[r]^{i \square 0}  \ar[u] &
M  \ar[r] \ar[u]^{\phi} \square X & X \ar@<.5ex>[l]^{t_\pm} \ar[u]^\varphi
}$$
such that
$\varphi (s_+ \square 1) \mu \Delta_{s_- \square 1} = t_+ \mu \Delta_{t_-}$
in $GK$.

We assume here that the $G$-action on $M_2(M \square A)$ is of the form $\theta \square (\alpha \otimes 1)$.
\end{lemma}

\begin{proof}
Let $X=(X,\gamma)$.
If the $M_2$-action of the first line is
$\theta \square (\alpha \otimes 1)$,
then of the second line put it to $\theta \square (\gamma \otimes 1)$.
Set
$\phi= {\rm id} \square \varphi$
and 
$t_\pm = \varphi s_\pm \square 1$
and check 
the claim with lemma \ref{lemma21} with $\Phi=\phi \otimes 1$.
%
%
\end{proof}

\begin{lemma}		\label{lemma14}
Every split exact sequence as in the first line
is isomorphic to the one of the second line
as indicated in this diagram:
$$\xymatrix{
B  \ar[r]^i \ar[d]  & M  \ar[d]^\phi  \ar[r]^f & A \ar@<.5ex>[l]^{s_\pm}  
\ar[d]   \\
B  \ar[r]^j  & \call_B(B) \square  A \ar[r]  & A \ar@<.5ex>[l]^{t_\pm} 
}$$
That is, 
$s_+ \mu \Delta_{s_-} = t_+ \mu \Delta_{t_-}$ in GK.



The $G$-action on $M_2(\call_B(B) \square A)$ is of the form ${\rm Ad}(S \oplus T) \square \delta$,  
where 
$(M_2(A),\delta)$ and $(M_2(\call_B(B)), {\rm Ad}(S \oplus T))$
are 
$G$-algebras.
\end{lemma}

\begin{proof}
Set $j=i \phi$.
Define $\chi$ analogously as in (\ref{chi}).
%
Set 
$\phi(m)= \chi(m) \square f(m)$.
Put $t_\pm(a)= \chi(s_\pm(a)) \square a$.
%
%
Note that $\phi$ is bijective.
Set $\Phi = \phi \otimes 1_{M_2}$ 
and
define the 
$G$-action on its range in such a way that
$\Phi$ becomes $G$-equivariant.
Verify with lemma
\ref{lemma21}.

For the sake of simpler notation we assume now that $i$ is the identity embedding.
By lemma \ref{lemma61}.(iv), $M_2(M)$
embedds equivariantly into $\call_M(M \oplus M,{\rm Ad}(S \oplus T))$, confer the formula in lemma
\ref{lemma51}.

But $S$ restricts to a $G$-action on $B$, and $T$ restricts to a Hilbert $(B,S)$-module on $B$ 
(because $T(b^2)= T(b) S(b) \in B$).

Hence by \ref{lemma61}.(vii)
we can equip $\call_B(B \oplus B)$
with the $G$-action ${\rm Ad}(S \oplus T)$ as well, or in other words, use the same formula as in lemma \ref{lemma51}.
The $G$-action $\delta$ 
comes from corollary \ref{cor79}. 
%
\end{proof}


\begin{definition}				\label{def75}
{\rm
Let $(B,\beta)$ be a $G$-algebra.
Write $\H_B := \bigoplus_{i=1}^\infty B$ for the infinite direct sum Hilbert $(B,\beta)$-module.
We shall equip $\H_B$ with various $G$-actions $S$, but 
often require that $S$ is of the form
$S=\beta \oplus T$ for a $G$-action $T$ on $\H_B$
($S_g(b_1 \oplus b_2 \oplus b_3 \oplus \ldots)=
\beta_g(b_1) \oplus T(b_2 \oplus b_3 \oplus \ldots)$).
The letter $R$ will always stand for such a $G$-action and 
we may pick out
$R_0:= \beta \oplus {\rm triv}$ deliberately. 
}
\end{definition}

If a copy of $\H_B$ is derived from
another construction, say
the Kasparov 
stabilization theorem then we always keep the original $G$-action:

\begin{lemma}				\label{lemma86}
Let $(\cale,S)$ and $(\H_B, \beta \oplus T)$ be $G$-Hilbert
$(B,\beta)$-modules.
Then there is a $G$-Hilbert $(B,\beta)$-module
isomorphism
$$Y:(\cale,S) \oplus (\H_B,\beta \oplus T) \rightarrow ( 
\H_B, \beta \oplus V)	.$$
\end{lemma}

\begin{proof}
Excluding the first coordinates $(B,\beta)$ of the $\H_B$s on which $Y$ is set to be the identity, 
we apply Kasparov's non-equivariant stabilization theorem to obtain $Y$, and define the $G$-action $V$ in such
a way that 
$Y$ becomes $G$-equivariant.
\end{proof}

\begin{definition}
{\rm
Define the $C^*$-algebra isomorphism
$$\kappa : B \otimes \calk
\rightarrow
\calk_B 
(\H_B) \subseteq \call_B 
(\H_B):
\kappa\big ((b_{ij} \otimes e_{ij})_{i,j}
\big ) \big((\xi_k)_k \big)
= \Big (\sum_j b_{ij} \xi_j \Big )_i$$
where $b_{ij}, \xi \in B$ and $1 \le i,j,k$.
We equip $B \otimes \calk$ with the $G$-action
such that $\kappa$ 
becomes equivariant.
That is, if $\gamma$ is the $G$-action on $\call_B(\H_B)$ then $\delta = \kappa^{-1} \circ \gamma \circ \kappa$ is the $G$-action on $B \otimes \calk$.

}
\end{definition}

\begin{definition}
{\rm
If the $G$-action on $\H_B$ is $\beta \oplus S$ then we have an $G$-invariant corner embedding $(B,\beta) \rightarrow (B \otimes \calk, \delta)$ which we denote by $e_\delta$ or $e$ if $\delta$ is understood.
}
\end{definition}

In lemma  \ref{lemma215} we saw how we can merge a homomorphism from the right hand side with a split exact sequence.
The next lemma is the analogy from the left hand side.

\begin{lemma}	\label{lp}

Let the first line of the following diagram be given, where $f$ denotes an equivariant $*$-homomorphism.
Then it can be completed to this diagram
$$\xymatrix{
C \ar[d]^{e} & B  \ar[l]_f \ar[r]^i 
& 
\call_B(B) \square A \ar[rr] 
\ar[d]^\phi  &&  A 
\ar@<.5ex>[ll]^{s_\pm \square} 
\ar[d]
\\
C \otimes \calk \ar[rr]^\kappa 
& 
& 
\call_C \big( B \otimes_f C \oplus 
\mathbb{H}_C
\big ) \ar[rr] \square A  && A \ar@<.5ex>[ll]^{t_\pm}
}$$
such that $(s_+ \square 1) \mu \Delta_{s_- \square 1} f e = t_+ \mu \Delta_{t_-}$
in $GK$.
\end{lemma}

\begin{proof}
Let $C=(C,\gamma)$.
We set 
$$\phi(T \square a)= T \otimes 1 \oplus 0 
\square a$$
and $t_\pm = (s_\pm \square 1) \phi$.
If the $M_2$-action of the first line 
is 
${\rm Ad}(S \oplus
T) \square \delta$
 (confer lemma \ref{lemma14}),
then we set it to
${\rm Ad}(S \otimes \gamma \oplus R_0,
T \otimes \gamma \oplus R_0) \square 
\delta$ in the second line.

We have a $G$-Hilbert $C$-module isomorphism
$$B \otimes_f C \rightarrow C_0:= \overline{ 
f(B) C} \subseteq C
: b\otimes c \mapsto f(b)c
$$
(norm closure of sums) into a $G$-Hilbert $C$-submodule $C_0$ of $C$.

We have an equivariant $*$-homomorphism
$$h:M_2(f (B)) \rightarrow \call_{C}(C_0 \oplus C) \subseteq \call_{C}(C_0 \oplus \H_C)$$
by matrix-vector multiplication,
where the summand $C$ means here the 
distinguished first coordinate $(C,\gamma)$ of $\H_C$.


That is why we can rotate 
$i \phi$
to $g$ for
$$g(b)= (0 \otimes 0 \oplus (f(b) \oplus 0)) \square 0$$
by a homotopy in 
the image of $h$, which is in
$\kappa(C \otimes \calk)$.

Thus $i \phi=g = f e \kappa$ in $GK$.
It is now easy to verify with lemma \ref{lemma21}.
%
%
%
\end{proof}

\begin{lemma}		\label{lemma83}

Given the upper right double split exact sequence
of this diagram 
one can draw this dagram
$$\xymatrix{
B \otimes \calk 
\ar[d]
&
\calk_B(\cale) \ar[l]_h \ar[r]^j
&
\call_B (\cale )  \square A \ar[rr] 
\ar[d]^\phi
&& A 
\ar@<.5ex>[ll]^{s_\pm \square } 
\ar[d]
\\
B \otimes \calk \ar[rr]^\kappa &&
\call_B (\cale \oplus \H_B
) \ar[rr] 
\square A && A \ar@<.5ex>[ll]^{t_\pm}
}$$
such that
the first line is the second line in $GK$, i.e. 
$(s_+ \square 1) \mu \Delta_{s_- \square 1} h = t_+ \mu \Delta_{t_-}$.

\end{lemma}

\begin{proof}
We set
$\phi(T \square a) = T \oplus 0 \square a$,
$t_\pm = s_\pm \oplus 0 \square {\rm id}$
and 
$h= j \phi \kappa^{-1}$.
If ${\rm Ad}(S \oplus T) \square 
\delta$ is the $M_2$-action of the first line, we put it to ${\rm Ad}(S \oplus R_0, T \oplus R_0) \square
\delta$ on the second line.
Recall that $B \otimes \calk$ has the $G$-action
$\kappa^{-1} \circ {\rm Ad}(S \oplus R_0) \circ \kappa$.
Verify with lemma \ref{lemma21}. 
\end{proof}


\begin{lemma}   \label{lemma91}

Let the first line of the following diagram be given and $\phi_t$ be evaluation at time $t \in [0,1]$. Then it 
can be completed to this diagram
$$\xymatrix{
B[0,1]		
\ar[d]^{\phi_t}   \ar[r]^e
&
B[0,1] \otimes \calk
\ar[r] 
\ar[d]^{a_t}
&
\call(\cale \oplus \H_{B[0,1]})  \square A \ar[rr] 
\ar[d]^{b_t} && A 
\ar@<.5ex>[ll]^{s_\pm} 
\ar[d]
\\
B  \ar[r]^{e_t} & 
B \otimes \calk \ar[r] 
&
\call(\cale \otimes_{\phi_t} B
 \oplus \H_B) \square A 
 \ar[rr] 
 && A \ar@<.5ex>[ll]^{s_\pm b_t} 
}$$

such that 
$s_+ \mu \Delta_{s_-} e^{-1} \phi_t
= s_+ b_t \mu \Delta_{s_- b_t} {e_t}^{-1}$
for all $t \in[0,1]$
and these elements do not depend on $t$.


\end{lemma}


\begin{proof}
Here $a_t$ and $b_t$ with $b_t(T \square a)=T \otimes_{\phi_t} 1 \square a$
are the evaluation maps at time $t \in [0,1]$.

Since $\phi_t$ is the evaluation of the identity homotopy on $B[0,1]$ in $C^*$
(which, note, $a_t$ 
is not)
$\phi_0 = \phi_t$ for all $t$.
 
If the $M_2$-action of the first line of the diagram is ${\rm Ad}(S \oplus R,T\oplus R) \square \delta $ 
then of the second line it is
${\rm Ad}((S  \oplus R)\otimes_{\phi_t} \beta,
(T\oplus R) \otimes_{\phi_t} \beta)
\square \delta$.  

Now verify the claim with lemma \ref{lemma21} and
$\Phi= b_t \otimes 1$.
%
%
\end{proof}

Normally, a homotopy runs in 
a fixed algebra with a 
fixed $G$-action. If we combine 
homotopy with matrix technique, we can however allow homotopies where the $G$-action of the range algebra, and so the range object changes:

\begin{lemma}		\label{lemma123}
Let $s: A \rightarrow \big ( M_2(X[0,1]), (\theta^{(t)})_{t \in [0,1]} \big )$
be an equivariant 
homomorphism into the lower right corner for $\theta$ as in lemma \ref{lemma61}.
Assume that the upper left corner action $\theta_{11}^{(t)}$ does not depend on $t \in [0,1]$.
Then $s_0 {e_{11}^{(0)}}^{-1} = s_1 {e_{11}^{(1)}}^{-1}:A \rightarrow X$ in $GK$.
\end{lemma}

\begin{proof}
Consider the diagram
$$\xymatrix{
\big (X[0,1], (\theta^{(t)}_{11})_{t \in [0,1]} \big )
\ar[rr]^{e_{11}}		\ar[d]^{\psi_t}
&&
\big ( M_2  (X[0,1]), (\theta^{(t)})_{t \in [0,1]} \big) 
\ar[d]^{\phi_t}
   	&&  A   	\ar[ll]^s		\ar[d]  	\\
\big (X, \theta^{(t)}_{11}  \big )
\ar[rr]^{e^{(t)}_{11}}
&&
\big ( M_2(X), \theta^{(t)}  \big ) 
&& A
\ar[ll]^{s_t}
}
$$
where $e_{11}$ and $e_{11}^{(t)}$ are the corner embeddings and $\phi_t$ and $\psi_t$ are the evaluation maps.
Since both rectangles of the diagram commute we get $s e_{11}^{-1} \psi_t = s_t {(e_{11}^{(t)})}^{-1}$.

For $\theta^{(t)}_{11}$ is independent of $t$, 
$\psi_t$ is
evaluation of the identity homotopy,
so $\psi_0 = \psi_t$ in $GK$.
%
\end{proof}

In lemma \ref{lemma272}.(iv) we have observed 
a $G$-equivariance condition reminiscent of $KK$-theory.
In the next lemma we are going to observe
how grading and the commutator condition
$[a,F] \in \calk(\cale)$ 
come into play:
Namely, 
if we start with a single split exact sequence, how can we construct a second split?:

\begin{lemma} 				\label{lemma812}

Let $U$ be a unitary in $M$
and $t : A \rightarrow M$ a $*$-homomorphism.
Consider a diagram 
$$\xymatrix{
B
\ar[r]^i 
& 
(M,\gamma) \square A \ar[rr]    && A 
\ar@<.5ex>[ll]^{s \square}
\ar[lld]^{U \circ t \circ U^* 
\oplus 1}		\\ 
 & (M,\delta) \oplus A &&
}$$



Then the image of $U \circ t \circ U^* \oplus 1$ is in $M \square A$ if and only if $s(a) \circ U - U \circ t(a) \in i(B)$ for all $a \in A$
if and only if
$$\Big [s(a) \oplus t(a), \left(\begin{matrix} 0 & U
\\  U^* & 0 \end{matrix} \right )   \Big ]
\in M_2(i(B))
\qquad \forall a \in A$$

\end{lemma}

\begin{proof}
The proof is straightforward, or see the similar proof of lemma \ref{lemma112}.
\end{proof}

If the condition of the last lemma is satisified, then $U \circ t \circ U^* \square 1$ gives us a second split in the first line of the diagram.

The next and final step to complete the first line to a double split exact sequence would be a $M_2$-action.

Typically $t$ is equivariant with respect to $\gamma$ and one defines the $G$-action on
$X:=M_2(M \square A)$ by $\theta^{U \square 1}$.
To this end $X$ must be invariant under this action, and for $M=\call(\cale)$
this equivalent to the other condition of Kasparov theory, see lemma
\ref{lemma272}. 

\section{The functor $\calb$}

\label{sec9}

%
For the following definition see for example
\cite[17.6]{blackadar}.

\begin{definition}			\label{def71}
{\rm
Let $z=[s,(\cale,S),F] \in KK^G(A,B)$ be a Kasparov element.
By functional calculus we choose an operator homotopy $(s,(\cale,S),F) \sim (s,(\cale,S),F')$ such that $F'$ is self-adjoint and $\|F'\|\le 1$.
We denote the new $F'$ by $F$ again.
Set 
$$U = 
\left (\begin{matrix} F  &  (1-F^2)^{1/2}\\
 (1-F^2)^{1/2}  &  -F
\end{matrix} \right ) 
 \in \call_B \big ((\cale,S) \oplus (\cale,
S) 
\big )$$


Since $U$ is a compact perturbation of $F \oplus (-F)$, 
by adding on zero cycles to $z$ we get
$$z=
[s \oplus 0 \oplus 0,(\cale \oplus \cale \oplus \H_B, S \oplus S 
\oplus R),U \oplus 1 
] =$$
\begin{equation}		\label{secz}
z = [s \oplus 0, 
(\cale \oplus \H_B, S \oplus 
V),
U(F)
]
\end{equation}
where we have 
written
$(\cale, S 
) \oplus (\H_B,R) \cong (\H_B, 
V)$
for simplicity,
and have set
$$U(F):=U \oplus 1 \qquad \in \call_B \big((\cale,S) \oplus (\H_B, 
V ) \big) ,$$
which is a self-adjoint unitary,
but 
notice that $U(F)$ 
means still the first $U \oplus 1$ operator.

}
\end{definition}

\begin{lemma}		\label{lemma81}
\begin{itemize}

\item[(i)]
If $(s,\cale,F_t)_{t \in [0,1]} \in KK^G(A,B)$ is an operator
homotopy, then $U(F_t)_{t \in [0,1]} \in \call_B(\cale \oplus \H_B)$
is homotopy of unitaries.

\item[(ii)]
If $F$ is self-adjoint unitary
and $s:A \rightarrow \call_B(\cale)$ a non-equivariant homomorphism, then 
$$U(F) \circ \big( s 
(a) \oplus 0 \big) \circ U(F)^* = F \circ s(a) \circ F^* \oplus 0$$

\item[(iii)]
If $B=D[0,1]$ and $\phi_t:D[0,1]\rightarrow D$
is the evaluation map at $t \in [0,1]$
then
$U(F) \otimes_{\phi_t} 1 = 
U(F \otimes_{\phi_t} 1)$.
\end{itemize}

\end{lemma}

\begin{proof}
These claims follow easily from definition
\ref{def71}. Recall that the transition
from $F$ to $F'$ respects homotopy.
\end{proof}


\begin{definition}		\label{defB}
{\rm
Let
$z=[s_- + s_+, (\H_B,S), F
] \in KK^G(A,B)$ be given
where $F$ is unitary
and $S$ is of the form $\beta \oplus T$, see definition \ref{def75}.

Then we define
\begin{displaymath}			\label{bz}
\calb(z)= t_+ \mu \Delta_{t_-} e^{-1}
\end{displaymath}
that is in details, the element of $GK$ associated to this diagram
read from right to left:
\begin{equation}	 	\label{bz2}
\xymatrix{
B \ar[r]^e &
B \otimes \calk
%
\ar[r]^\kappa 
& 
\call_B( (\H_B,S)) \square A \ar[r]    & A
\ar@<.5ex>[l]^{t_\pm} 
}
\end{equation}
where
\begin{eqnarray*}
t_-(a) &=& s_-(a) \square a \\
t_+(a) &=& F \circ s_+(a) 
\circ F^* \square a
\end{eqnarray*}

where
the $G$-action on $M_2( \call_B( (\H_B,S)) \square A)$ is
$\theta^{F \square 1}$.  
The letter $e$ denotes the equivariant corner embedding.

}
\end{definition}

Here, $\theta^{F \square 1}$ 
is an incorrect but suggestive notation
for $\theta^F \square (\alpha \otimes 1_{M_2})$.

We note that the action $\theta^{F \square 1}$
is just ${\rm Ad}(S \oplus T) \square (\alpha \otimes 1)$, where the $G$-action $T$ is defined in such a way that the unitary $F$ becomes equivariant, that is, $T \circ F = F \circ S$.

\begin{lemma}		\label{lemma112}
The assignment $\calb$ is well defined.
\end{lemma}

\begin{proof}
{\em 
The image of $t_+$.}
Since $F$ is odd graded, $F=\left (\begin{matrix} 0 & V \\ V^* & 0 \end{matrix} \right )$
for a unitary $V$ acting on the Hilbert module $\H_B \cong \cale_- \oplus \cale_+$ 
(graded parts).
Let us also write $s_-(a)= v_-(a) \oplus 0$
and $s_+(a)= 0 \oplus v_+(a)$ 
acting on $\cale_- \oplus \cale_+$.

Because $[s_-(a) + s_+(a),F]$ is a compact operator in Kasparov theory, we get that
$$[v_-(a) \oplus v_+(a), F]  F^*
= (v_-(a) \oplus v_+(a) )   F  F^*-
F   (v_-(a) \oplus v_+(a) ) F^*$$
$$
= v_-(a) \oplus v_+(a) - V v_+(a) V^* \oplus V^* v_-(a) V$$
is in $\calk_B(\H_B)$, so the first coordinate
$v_-(a)  - V v_+(a) V^*$ is in $\calk_B(\cale_-)$. But this means
$s_-(a)  - F s_+(a) F^* \in \calk_B(\H_B) 
$.
Hence 
$$ 
t_+(a) \in t_-(A) + \kappa(B \otimes \calk) = \call_B(\H_B) \square A$$


{\em $M_2$-action.}
Define 
a $G$-action $T$ on $\H_B$
by $T_g \circ F = F \circ S_g$ .
By lemma \ref{lemma272}.(iii)-(iv),
the $G$-action $\theta^{F \square 1}$
is valid.

{\em Corner embedding.}
If we had two choices of corner embeddings $e$ then they 
would be the same in $GK$ by a rotation homotopy.

{\em Homotopy invariance.}

Let
$z=[s_- 
+ s_+, (\H_{B[0,1]}, S)
,F] \in KK^G(A,B[0,1])$ be 
a (operator) homotopy where the end point operators $F_0,F_1$ are unitary.
We rewrite $z$ in the form (\ref{secz})
and go with it into definition \ref{defB}.
%
Then applying the diagram (\ref{bz2}) 
associated to $\calb(z)$ to 
lemma \ref{lemma91}
we see by lemma \ref{lemma81}.(ii)-(iii) that $\calb(z_0)=\calb(z_1)$
for the evaluations $z_t$ of $z$ at time $t \in [0,1]$.
\end{proof}


\begin{lemma}			\label{lemma122}
Let
$z=[s_- 
+ s_+, (\cale,S)
,F] \in KK^G(A,B)$ be given.

Then 
$\calb(z)= t_+ \mu \Delta_{t_-} e^{-1}$
for the diagram
$$\xymatrix{
B \ar[r]^e &
B \otimes \calk
%
\ar[r]^\kappa 
& 
\call_B( 
\cale \oplus \H_B) \square A \ar[r]    & A
\ar@<.5ex>[l]^{t_\pm} 
}$$
where
$\cale \oplus \H_B$ is equipped with 
a $G$-action
$S \oplus R$ and
the $G$-action on $M_2( \call_B( \cale \oplus \H_B) \square A)$ is defined to be
$\theta^{U(F) \square 1}$  
%
and
\begin{eqnarray*}
t_-(a) &=& s_-(a) \oplus 0\square a \\
t_+(a) &=& U(F) \circ (s_+(a) \oplus 0
) \circ U(F)^* \square a
\end{eqnarray*}
\end{lemma}

\begin{proof}
We bring $z$ to the form (\ref{secz}) 
and apply definition \ref{defB}.
\end{proof}

\begin{lemma}		\label{lemma124}
Consider lemma \ref{lemma122}
where $z=[s_- + s_+, (\cale,S)
,(F_m)_{m \in [0,1]}]$
is an operator homotopy.
Then $t_+^{(0)} \mu = t_+^{(1)} \mu$
in $GK$
for
$$t_+^{(m)}(a) = U(F_m) \circ (s_+(a) \oplus 0
) \circ U(F_m)^* \square a
$$
\end{lemma}

\begin{proof}
Set $X= \call_B( \cale \oplus \H_B) \square A$.
We consider $M_2(X[0,1],(\theta^{F(U_m)\square 1})_{m \in [0,1]})$.

Now the claim follows by lemma \ref{lemma123}.
\end{proof}

\section{
Connection between $\cala$ and $\calb$}

\label{sec10}

\begin{definition}			\label{def161}
{\rm
Let $(B \otimes \calk,\gamma)$ be a $G$-algebra.
Let $E = B \otimes e_{ii}\cong (B,\beta) \subseteq B \otimes \calk$
be a $G$-invariant corner algebra.
Then 
$((B \otimes \calk) E,\gamma)$ is a $G$-Hilbert $(B,\beta)$-module 
with all operations inherited from 
the $G$-algebra 
$B \otimes \calk$.
%
%
%
We define
$$r 
:\big (\call_{B \otimes \calk}
(B \otimes \calk), {\rm Ad}(\gamma) \big )
 \rightarrow 
\big ( \call_B(\H_B), {\rm Ad} (Y \circ \gamma
\circ Y^{-1} ) \big ) : 
r 
(T) (\xi) = Y(T Y^{-1}(\xi))$$
($T$ acts here by multiplication) 
where the $G$-Hilbert $B$-module isomorphism
$$Y: 
(B \otimes \calk) E
\rightarrow
\H_B 
$$
is the canonical map by regarding
elements of the domain of $Y$ as column vectors.
The $G$-action on $\H_B$ is
$Y \circ \gamma
\circ Y^{-1}$.
}
\end{definition}

\begin{lemma}			\label{lemma162}
$r$ and
$$
\zeta :\call_B(\H_B) \rightarrow \call_{B \otimes \calk}(B \otimes \calk) :\zeta(T)(x)= \kappa^{-1}(T \kappa(x))$$
($T$ acts by multiplication)
are inverse equivariant $*$-homomorphisms to each other.


\end{lemma}

\begin{lemma}
Let $\theta$ be as in lemma
\ref{lemma61} and $\mu_\theta = e_{22} e_{11}^{-1}: A \rightarrow A$ 
%
%

Then $\cala(\mu_\theta)= [{\rm id}_A, (A,\gamma), 0]$.
\end{lemma}

\begin{proof}
Let $A_{11}$ and $A_{22}$ be the corner $G$-subalgebras of $M_2(A)$.
Then 
\begin{eqnarray*}
\cala(\mu_\theta) &=& \cala(e_{22}) \cala(e_{11}^{-1})  = [(A_{22} A,\theta),0]
\cdot [ (A A_{11},\theta),0] 	\\ 
&= & [(A_{22} A,\theta) \otimes_A
(A A_{11},\theta),0] = [(A,\gamma),0]
\end{eqnarray*}

\end{proof}

\begin{lemma}		\label{lemma121}
Let the right part of the first line of the following diagram
be double split:
$$\xymatrix{
B \ar[r]^e  \ar[d] &
B \otimes \calk \ar[r]^j  \ar[d]^\psi  &
\ar[r]^f 
\ar[d]^\phi 
X & A \ar@<.5ex>[l]^{u,h} \ar[d] 
\\
B \ar[r]  &
B \otimes \calk  
\ar[r]^\kappa 
& 
\call_B(\H_B^2) \square A \ar[r]    & A 
\ar@<.5ex>[l]^{s_\pm} 
}$$

Then
$$
\cala(h \mu \Delta_{u} e^{-1})
=$$
$$[u \chi r \oplus h  \chi r,  (\H_B,Y \circ j^{-1} \circ \gamma \circ j \circ Y^{-1}) \oplus 
(\H_B,Y \circ  j^{-1} \circ \Gamma \circ j \circ Y^{-1}) , F ]$$

where $F$ is the flip, $\chi$ is like (\ref{chi}), 
$r$ and $Y$ are from definition \ref{def161},
and $\left ( \begin{matrix} \gamma & \Gamma' \\
\Gamma &  \delta \end{matrix}\right )$
is the $G$-action on $M_2(X)$.


\end{lemma}

\begin{proof}
Set $I= B \otimes \calk$.
For simpler notation we assume that $I$ is embedded in $X$ in the diagram.
The $G$-action on $M_2(I)$ is denoted in the same way as 
the one on $M_2(X)$.
We have two isomorphisms
%
$$
V:(A \otimes_h X \otimes_\chi I,
\alpha \otimes \Gamma \otimes \gamma)
\rightarrow
(A \otimes_{h \chi} I,
\alpha \otimes \Gamma) 
: V(a \otimes x \otimes i)= a \otimes x i  \\
$$
$$
W:(A \otimes_h X \otimes_{f u \chi} I,
\alpha \otimes \gamma \otimes \gamma)
\rightarrow
(A \otimes_{u \chi} I,
\alpha \otimes \gamma) 
: W(a \otimes x \otimes i)= 
a \otimes u(f(x)) i
$$
of $G$-Hilbert $I$-modules.

%

In the following computation $F$ stands always for a flip operator (on possibly different spaces). Notice that
$0 \# F$ is a just a $F$-connection.
Recall from \cite[prop. 9.(f)]{skandalis}  the associativity of $F$-connections.
We compute
$$
\cala (h \mu \Delta_u)= 
\cala(h) 
\cala(\mu) 
\cala(\Delta_u)=
h^*([\mbox{id},(X,\Gamma),0] \cdot [fu \chi \oplus \chi,(I \oplus I , \gamma \oplus \gamma),F])
$$
$$=
[\mbox{id} \oplus \mbox{id}, A \otimes_h X \otimes_{fu \chi} I \oplus 
A \otimes_h X \otimes_{\chi} I,0 \# 0 \# F]$$
$$=
[\mbox{id} \oplus \mbox{id}, (A \otimes_{u \chi} I, \alpha \otimes \gamma) \oplus 
(A \otimes_{h \chi} I, \alpha \otimes \Gamma), 0 \# F]$$
$$=
[u \chi \oplus h \chi, (I \oplus I, \gamma \otimes \Gamma), F]$$

Multiplying this 
with 
$
\cala (e^{-1})=[(I E,\gamma),0] \in 
KK^G((I,\gamma),(B ,\beta))$ this gives
$$
\cala (h \mu \Delta_u e^{-1})
=
[u \chi \oplus h  \chi,  (I, \gamma) \otimes_{(I,\gamma)} (I E,\gamma) \oplus 
(I,\Gamma)  \otimes_{(I,\gamma)} (I E,\gamma) , F \otimes 1]$$
$$=
[u \chi \oplus h  \chi,  (I E \oplus I E,\gamma \oplus \Gamma) , F ]$$
$$=
[u \chi r \oplus h  \chi r,  (\H_B \oplus \H_B , Y \circ \gamma  \circ Y^{-1} \oplus  Y \circ \Gamma \circ Y^{-1} ) , F ]$$

If we had allowed $j$ to be general, then
the $G$-actions on $I$ would have been
$j^{-1} \circ \Gamma \circ j$ and
 $j^{-1} \circ \gamma \circ j$.
\end{proof}

\begin{lemma}		\label{lemma165}
Consider the last lemma.
{\rm (i)} Then 
$$
\calb( \cala(h \mu \Delta_{u} e^{-1})) = h \mu \Delta_{u} e^{-1}$$
provided $\delta = \alpha \otimes 1$ in corollary
\ref{cor79}.

{\rm (ii)} The first line is the second line in $GK$ of the diagram of the last lemma.
 
That is, $h \mu \Delta_u e^{-1} = s_- \mu \Delta_{s_+}
e^{-1}$.
\end{lemma}

\begin{proof}
(i)
If we put the computed $KK$-element of the last lemma into $\calb$, we get exactly the second line of the last lemma as follows:

(ii)
We complete the diagram of the last lemma by setting $\psi = j \phi \kappa^{-1}$ and 
\begin{eqnarray*}
\phi(x)  &=&  \chi r(x) \oplus 0 \square 
f(x)  \\
s_-(a) &=& u \chi r(a) \oplus 0 \square a \\
s_+(a) &=& h \chi r(a) \oplus 0 \square a
\end{eqnarray*}

We equip $D:=\call_B(\H_B^2) \square A$
with the $G$-action ${\rm Ad}(S) \square \alpha$, 
where $S$ is the obvious $G$-action notated in the cylce of the last lemma.
We define the 
$G$-action of $M_2(D)$ to be 
$\theta^{F \square 1}$.
The first line of the diagram is the second line  by lemma 
\ref{lemma21} 
for $\Phi= \phi \otimes 1$.
Thereby we
note that 
$$\chi r (x) = Y \circ j^{-1} \circ m_x \circ j \circ Y^{-1}$$
where $m_x$ is multiplication with $x$, such that
$$
Y \circ j^{-1} \circ \Gamma_g \circ j \circ Y^{-1}
\circ
\chi r (x) \circ Y \circ j^{-1} \circ \gamma_{g^{-1}} \circ j \circ Y^{-1}
= \chi r (\Gamma_g(x))$$
which shows equivariance of $\Phi$ in the lower left corner, which is sufficient by corollary \ref{cor74}.

%

For general $\delta$ we define the $M_2$-action of the second line of the diagram in such way that the bijective map $\Phi$ is equivariant.
\end{proof}

\begin{lemma}    \label{lemma106}   
$\cala \circ \calb = {\rm id}$.
\end{lemma}

\begin{proof}
We 
replace the
first line of the diagram of lemma \ref{lemma121}
by the diagram of definition \ref{defB}.
So we have $j=\kappa$, and thus by lemma \ref{lemma162},
$\chi r = {\rm id}$.

Let ${\rm Ad}(S \oplus T) \square (\alpha \otimes 1)$ be the $M_2$-action of the first line.

Then 
the action on $\H_B$ of the second line computes as follows:
Note that $\Gamma(x) = T_g \circ x  \circ S_{g^{-1}}$.
 
We take a column vector $Y^{-1}([b_{i1}]_i) \in (B\otimes \calk)E$, go with it into $\kappa$, and apply the action $\Gamma$ there and see what happens:
$$\Gamma_g(\kappa(Y^{-1}([b_{i,1}])) ([\xi_j])
= T_g( [b_{i,1}] \cdot S_{g^{-1}}([\xi_j]))
= T_g([b_{i,1} \beta_{g^{-1}}(\xi_1)])
= T_g([b_{i,1}]) \xi_1 $$
Thus we get
$$\Gamma_g(\kappa (Y^{-1}( [b_{i,1}]))) =  \kappa (Y^{-1} (T_g[b_{i,1}]))$$
or
$$Y \circ \kappa^{-1} \circ \Gamma_g \circ \kappa
\circ Y^{-1} = T_g$$
Analogously, $Y \circ \kappa^{-1} \circ \gamma_g \circ \kappa
\circ Y^{-1} = S_g$.

If $z=[s_- \oplus s_+, (\H_B^2,S \oplus V), H]$
with $H$ the flip, then $T= V \oplus S$,
and by lemma \ref{lemma121} we have shown
$$\cala (\calb(z))=
[s_- \oplus 0 \oplus s_+ \oplus 0,
(\H_B^4, S \oplus V \oplus V \oplus S), F]
= z$$

\end{proof}


\section{Preparation for pushout construction}

\label{sec11}

We remark that the proofs given in this section
also follow directly from known results in $KK$-theory and an application of lemma \ref{lemma165}.(i).
We still recall the proofs in the framework of $GK$-theory because it is an important technique and the proof is not so long.

For two 
$*$-homomorphisms
$s_\pm : A \rightarrow \call(\cale)$ 
and a Hilbert $A$-module $\calf$ we write
$$\calf \otimes_{s_\pm} \cale :=
\calf \otimes_{s_-} \cale
\oplus 
\calf \otimes_{s_+} \cale
\cong \calf \otimes_{s_- \oplus s_+} (\cale \oplus \cale)$$

\begin{lemma}		\label{lemma271}
Let the first line of the following diagram be given. Then it can be completed to this diagram
$$\xymatrix{
B  \ar[r]^e \ar[d] &
B \otimes \calk \ar[r]
\ar[d]
&
\call( \cale  \oplus \H_B 
) \square A 
 \ar[rr]
\ar[d]^\phi
 && A \ar@<.5ex>[ll]^{s_\pm \oplus 0 \square} 
\ar[d]
\\
B \ar[r]  & 
B \otimes \calk \ar[r]		
&
\call( 
\tilde A \otimes_{s_\pm} \cale
\oplus \H_B 
) \square A 
 \ar[rr]	
 && A \ar@<.5ex>[ll]^{u_\pm}   
\\
%
B[0,1]  \ar[r] \ar[u] \ar[d]  &
B[0,1] \otimes \calk \ar[r]		\ar[u]		\ar[d]
& \call \Big ( Z  \otimes_{y_\pm}
\cale [0,1]   \oplus \H_{B[0,1]}
\ar[u]
 \Big )  \square A \ar[rr] \ar[d]^{b_1} 
\ar[u]_{b_0}
&& A 
\ar@<.5ex>[ll]^{t_\pm} 
\ar[d]	\ar[u]
\\
B \ar[r] &
B \otimes \calk \ar[r]		
&
\call( 
A \otimes_{s_\pm} \cale
\oplus \H_B 
) \square A 
 \ar[rr]
 && A \ar@<.5ex>[ll]^{v_\pm}
}$$
such that the first line is the 
last line in GK,
i.e. $(s_+ \oplus 0 \square 1) \mu 
\Delta_{s_- \oplus 0 \square 1} e^{-1}= v_+ \mu 
\Delta_{v_-}  e^{-1}$.
Thereby 
\begin{eqnarray*}
v_-(a) &=&  
 a \otimes 1  \oplus 0\oplus 0 
\square a 			\\
v_+(a) &=& 
U(H) \circ (0 \oplus a \otimes 1  \oplus 0 ) \circ U(H)^* \square a    
\end{eqnarray*}
where $H$ is 
a $F$-connection
on
$A \otimes_{s} (\cale \oplus \cale)
\cong A \otimes_{s_\pm} \cale$
for the flip operator $F$ on $\cale \oplus \cale$. 

We assume that the $G$-action on $M_2(\call(\cale \oplus \H_B) \square A)$ of the given first line of the diagram is 
${\rm Ad}(S \oplus R ,  T \oplus R) \square (\alpha \otimes 1)$.
The $G$-action on $M:=\call(A \otimes_{s_\pm} \cale \oplus \H_B) \square A$ of the last line of the diagram then is 
$${\rm Ad}( 
\alpha  \otimes S  \oplus  \alpha \otimes T
\oplus R
 ) \square \alpha$$
and on
$M_2(M)$
it is 
$\theta^{U(H) \square 1}$.


\end{lemma}



\begin{proof}


Let $C([0,1])$ have the trivial $G$-action.
Let $(\tilde A,\tilde \alpha)$ be the unitization of $(A,\alpha)$, see \cite[def. 3.3]{buniversal} for inverse semigroups and \cite{legall} for groupoids.
Set
\begin{eqnarray*}
Z
&=& \{f \in \tilde A[0,1]|\, f(1) \in A\} \subseteq (\tilde A,\tilde \alpha) \otimes C([0,1])		\\
\cale[0,1] &:=& \cale \otimes C([0,1])
\end{eqnarray*}
Let $\tilde s_+, \tilde s_-: \tilde A \rightarrow \call(\cale)$ be the natural extensions
of $s_\pm$, and set
$$
y_\pm: \tilde A [0,1] \rightarrow \call(\cale[0,1])  :
y_\pm(x \otimes f) = \tilde s_\pm(x) \otimes f 
$$

Let $F$ be the flip operator on 
$\calf:=\cale[0,1] \oplus \cale[0,1]$.
Write
$$\calg= Z \otimes_{y_\pm} \cale[0,1]
\cong Z \otimes_{y_- \oplus y_+} \calf
=: \calh.$$

The $G$-action on $\calg$ is 
$$ 
(\tilde \alpha \otimes 1)  \otimes (S \otimes 1) \oplus (\tilde \alpha \otimes 1) \otimes (T \otimes 1)$$
and correspondingly the $G$-action on $\calh$ by the last isomorphism.

Choose a $F$-connection $V \in \call_B(\calh)$
on $\calh$
by 
\cite[lemma 2.7]{kasparov1988}.

Since $a \in \calk_A(Z) \otimes 1 \subseteq \call_B(\calh)$ for all $a \in A$,
by 
\cite[lemma 2.6]{kasparov1988}, $[a,V] \in \calk_B(\calh)$.
Similarly, by \cite[Lemma 10]{bsemimultikk}
$$a g(V) - a g(1) V \in \calk_B(\calh)$$
for all $a \in \calk_A(Z) \otimes 1$.

By 
\cite[Prop 9.(h)]{skandalis}, $z= [ 
{\rm id},\calh,V] \in KK^G(A,B)$.
Go with $z$ into lemma \ref{lemma122}
and create the third line of the diagram with it.
That is we set
\begin{eqnarray*}
t_-(a) &=& (a \otimes 1 \oplus 0 \oplus 0) \square a  \\
t_+(a) &=&
U(V)(0 \oplus a \otimes 1 \oplus 0) U(V)^* \square a
\end{eqnarray*}
The 
$M_2$-action of the third line of the diagram is $\theta^{U(H) \square 1}$.

The second and the fourth line of the diagram are the evaluations of the third line at time zero and one as written in lemma \ref{lemma91}.

By lemma \ref{lemma81}.(iii) we have 
$U(V) \otimes_{\phi_t} 1 = U(H)$ for $H:=V \otimes_{\phi_1} 1$.
Note that we get $v_\pm$ as claimed.

Completely analogously are $u_\pm$ defined
for $H' := V \otimes_{\phi_0} 1$.
Note that $H'$ is a $F'$-connection on 
$$\big (\tilde A \otimes_{s_- \oplus s_+} (\cale \oplus \cale), \tilde \alpha \otimes (S \oplus T) \big ) \cong
(\cale \oplus \cale,  S \oplus T)$$
where $F'$ is the flip on $\cale \oplus \cale$. Recall that there is a operator homotopy
$({\rm id},\cale \oplus \cale,H') \sim 
({\rm id},\cale \oplus \cale,F')$.
By lemma \ref{lemma81}.(ii) we can replace $U(H')$ by $F'$
in the definition of $u_+$.
Under identification of the last isomorphism, define
$$\phi(x \square a)= x \oplus 0  \square a$$
where $0$ is the operator on $\cale \cong  
\tilde A \otimes_{s_+} \cale$.
Verify the identity in $GK$ of line one and two
of the diagram with lemma \ref{lemma21}.
%
%
%
\end{proof}

\begin{lemma}			\label{lemma182}
The last lemma still holds true if we choose any
$F$-connection $H$.
\end{lemma}

\begin{proof}


Let $H \in \call(\calf)$ be any $F$-connection.
%
Consider the 
unit $1_A:=[id_A ,A ,0] \in KK^G(A,A)$.
Set $\calh:=A  \otimes_{s_+ \oplus s_-} (B \oplus B)$.
Then $[id_A , \calh , H]$
is a 
Kasparov product $1_A \otimes_A z$ because
$[a ,H] \in \calk(\calh)$ for all $a \in A$ by 
\cite[prop. 9.(e)]{skandalis},
$H$ is a $F$-connection, and $a [0,H] a^* = 0 \ge 0 \mod \calk(\calh)$, and 
so \cite[def. 2.10]{kasparov1988}
applies.


Hence
by the uniqueness of the Kasparov product,
$(id_A \oplus 0, \calh,H_1)$ and
$(id_A \oplus 0, \calh,H_2)$ are operator homotopic by Skandalis 
\cite[12.(a)]{skandalis} for any two $F$-connections $H_1$ and $H_2$.

Skandalis' proof is non-equivariant, but it works also equivariant,
see for example \cite{bsemimultikk} for inverse semigroups $G$.

Hence the definition in $GK$ of $v_+ \mu$ in lemma \ref{lemma271} does not depend on 
$H$ as one can connect different choices by a homotopy in the sense of lemma 
\ref{lemma124}.
\end{proof}

\section{Fusion with a
synthetical split}			

\label{sec12}



The following proposition shows that a composition of a double split exact sequence with a synthetical split yields a double split exact sequence again.

\begin{proposition}	\label{lemma4}

Let the first line of the diagram of lemma
\ref{lemma271} be given, and consider its fourth
line, (equivalently) rewritten down in the first line of the next diagram (without $e^{-1}$).

Let the right column of the 
next diagram
be a given split exact sequence.
Then we can complete 
these data to the following diagram
$$\xymatrix{
B \otimes \calk 
\ar[rr] 
\ar[d]
&& 
\call_B\Big ( \big(A \otimes_{s_-} \cale \oplus A \otimes_{s_+} \cale \big )^2 \oplus \H_B \Big) \square A \ar[rr] \ar[d]^{\phi}  && A 
\ar@<.5ex>[ll]^{c_\pm} 
\ar[d]^j  
\\
B \otimes \calk \ar[rr] &&
\call_B \Big (  \big (A \otimes_{s_-} \cale \oplus A \otimes_{s_+} \cale  \big)^2 \oplus \mathbb{H}_B
\Big ) \ar[rr] \square X && X \ar@<.5ex>[ll]^{t_\pm}  
\ar[d]^g
\\
&&&& Q 
\ar@<.5ex>[u]^u
}$$

such that 
$$\Delta_u ((s_+ \oplus 0) \square 1) \mu \Delta_{s_- \oplus 0 \square 1} e^{-1}
= \Delta_u c_+ \mu \Delta_{c_-} e^{-1}
= t_+ \mu \Delta_{t_-} e^{-1}$$

\end{proposition}

\begin{proof}
Set
$$\calf' := A \otimes_{s_-} \cale \oplus A \otimes_{s_+} \cale$$
and $\calf:=\calf' \oplus \calf'$ with the 
imagined grading $- \oplus +$.
We regard $\calf \oplus \H_B$ as having five summands (the four $A \otimes_{s_\pm} \cale$ and one
$\H_B$).

We define the Kasparov cycle
$$z = [s_- \oplus s_+, \cale \oplus \cale, T] \in KK^G(A,B),$$
%
where 
$\cale \oplus \cale$ has the 
obvious grading and $T$ is the flip operator.
As in definition \ref{def21}.(iii)
we set
$$
w = [gu \chi \oplus  \chi, A \oplus A, V] \in KK^G(X,A),$$
where $A \oplus A$ has the obvious grading and $V$ is the flip operator.

We form the Kasparov product $w \otimes_A z$
and obtain a cycle
$$w \otimes_A z = [v_- \oplus v_+, \calh,F] \in KK^G(X,B),$$
and using a canonical isomorphism,
which we are going to sloppily use
as an identification,
$$\calf \cong (A \oplus A) \otimes_{s_- \oplus s_+} (\cale \oplus \cale) =: \calh$$
(via 
$\oplus_{ (i,j)=(0,1),(1,0),(1,1),(0,0)
} a_i \otimes \xi_j \mapsto (a_1 \oplus a_0) \otimes (\xi_1 \oplus \xi_0)$)
we have
\begin{eqnarray*}
v_-(x) &=& \big( \chi(x) \otimes 1 \big) \oplus \big(gu  \chi(x) \otimes 1 \big) \oplus 0 \oplus 0 \\
v_+(x) &=& 0 \oplus 0 \oplus \big ( gu \chi(x) \otimes 1 \big) \oplus \big(
 \chi(x) \otimes 1 \big) 
\end{eqnarray*}
for all $x \in X$ (right hand side are operators on $\calf$).
We set 
\begin{eqnarray*}
c_-(a) &=& 
(a \otimes 1) \oplus 0 
\oplus 0  
\oplus 0 \oplus 0\square a\\
c_+(a) &=& U(F) \circ \big (0 \oplus 0 \oplus 0  \oplus (a \otimes 1)  \oplus 0 \big ) \circ U(F)^* \square a  \\
t_-(x) &=& v_-(x) \oplus 0   \square x  \\
t_+(x) &=& U(F) \circ \big( v_+(x)
\oplus 0 \big ) \circ U(F)^*
\square x
\end{eqnarray*}

By lemma \ref{lemma122}, the second line of the diagram is double split with $M_2$-action $\theta^{U(F) \square 1}$.

Recall that $F$ is a $T$-connection on $\calh$.
Applying this to
Lemma \ref{lemma271} for $A \oplus A$ instead of $A$, $\cale \oplus \cale$ instead of $\cale$, and 
$s_+ \oplus s_+ \oplus 0\square 1,
s_- \oplus s_- \oplus 0\square 1
$ instead of $s_\pm \oplus 0 \square 1$,
together with lemma \ref{lemma182},
and using lemma \ref{lemma51} two times, first in the first line of the diagram of lemma \ref{lemma271} and second in the fourth line, for
$A \oplus A$ instead of $A$, $X=A$, 
$\phi :A \rightarrow A \oplus A$ the 
injection onto the first corrdiante,
%
and using a similar assertion as in lemma \ref{lemma83}
that we may deliberately add summands without changing anything,
%
%
we obtain that
$$
c_+ \mu \Delta_{c_-} e^{-1} =
(s_+ \oplus 0 \square 1) \mu \Delta_{s_- \oplus 0 \square 1} e^{-1}$$ 

Also note that we have shown that the first line of the diagram is double split with $M_2$-action $\theta^{U(F) \square 1}$.

Consider the (null) cycle
$$
\sigma:=[guv_- \oplus guv_+,
\calh,V \otimes 1 ]\in KK^G(X,B)$$
(null because $[gu v_- \oplus gu v_+,V \otimes 1]=0$).
%
%

Recall that
$$x [V \otimes 1,F] x^*\ge 0 \mod \calk(\calf)$$ for all $x \in X$, and in particular for all $x=gu(x)$, by 
the definition of the Kasparov product
\cite[def. 2.10.(c)]{kasparov1988}.

Hence, by 
\cite[lemma 11]{skandalis}
(or \cite[lemma12]{bsemimultikk}
in the inverse semigroup equivariant setting), applied to 
$\sigma$ 
and
$$\sigma_2 := (gu)^*(w \otimes_A z) = [gu v_- \oplus gu v_+, \calh,F] \in KK^G(X,B),$$
$\sigma$ is operator homotopic to
$\sigma_2$.
%

Therefore,
$gut_+ \mu = gut_+' \mu$
by lemma \ref{lemma124},
where
\begin{eqnarray*}
t_+'(x) &:=& U(V \otimes 1) \circ \big( v_+(x)
\oplus 0 \big ) \circ U(V \otimes 1)^*
\square x  \\
   &=& (V \otimes 1) \circ  v_+(x)
 \circ (V \otimes 1)^* \oplus 0
\square x 		
\end{eqnarray*}
in $GK$ by lemma \ref{lemma81}.(ii), where we have identified $\calf \cong \calh$ for simplicity.
Hence
$$g u t_+ \mu \Delta_{t_-} = g u t_- \mu \Delta_{t_-}  = 
g u t_- \Delta_{t_-} = 0$$
where we have achieved $\mu=1$ by a rotation homotopy.

We define
$$\phi (x \square a ) = x \square j(a)$$

We verify all conditions of lemma \ref{lemma21}  for the first line and second line of the diagram
and for $\Phi = \phi \otimes 1$.
We then obtain
$j t_+ \mu \Delta_{t_-} = c_+ \mu \Delta_{c_-}$.
%
%
Thus we get
$$t_+ \mu \Delta_{t_-} = 
1_X t_+ \mu \Delta_{t_-} = (\Delta_u j + gu) t_+ \mu   
\Delta_{t_-} 
= \Delta_u c_+ \mu \Delta_{c_-}$$

\end{proof}

\section{Fusion with 
the inverse of a corner embedding}

\label{sec13}


In the following lemma we 
turn the composition of a double
split exact sequence with the inverse of a corner embedding 
to another double split exact sequence.

\begin{lemma}	\label{lemma5}

Let the first line of the diagram of lemma
\ref{lemma271} be given, and consider its fourth
line, partially written down in the first line of the next diagram.

Let $1 \le n \le \infty$ and set $M_\infty(A):= 
A \otimes \calk$. Let $f$ be a corner embedding.

Then we can complete these data to the
following diagram
$$\xymatrix{
B \otimes \calk 
\ar[r]^\kappa 
\ar[d]^l
& 
\call_B\Big ( 
A \otimes_{s_\pm} B 
\oplus \H_B \Big) \square A \ar[rr] \ar[d]^{\phi}  && A 
\ar@<.5ex>[ll]^{v_\pm} 
\ar[d]^f
\\
B \otimes \calk \otimes M_n \ar[r] &
\call_B \Big (  
M_n(A) E \otimes_{s_\pm} \cale 
\oplus \mathbb{H}_B^n
\Big ) \ar[rr] \square M_n(A) && M_n(A)  \ar@<.5ex>[ll]^{t_\pm} 
}$$


such that
$f^{-1} s_+ \mu \Delta_{s_-} = f^{-1} v_+ \mu \Delta_{v_-}  = t_+ \mu \Delta_{t_-}$.

\end{lemma}

\begin{proof}
Let $E=f(A)$ be the corner algebra (the upper left corner say).
The map 
$l$ is the equivariant corner embedding.
%

Recall that the $G$-action on $\calf := A \otimes_{s_\pm} \cale
\oplus \H_B$
is the canonical one 
induced by the actions of $(A,\alpha)$, $(\cale,S)$ and $(\H_B,R)$.
The $G$-action on $M_2(\call_B(\calf) \square A)$
is $\theta^{U(H) \square 1}$.

Denote the $G$-action on $M_n(A)$ 
by $\gamma$.

Since
$M_n(A) E \cong 
A^n$ (column vectors)
as non-equivariant Hilbert $A$-modules,
by just 
reordering summands 
we 
may consider the canonical isomorphism
$$Y: M_n(A) E  \otimes_{s_\pm} \cale  \oplus 
\H_B^n   
=: \calg
\rightarrow 
(A  \otimes_{s_\pm} \cale \oplus \H_B)^n 
$$
of non-equivariant Hilbert $B$-modules.
Recalling $U(H)$ from lemma
\ref{lemma271},
put 
$$V=
Y^{-1} \circ U(H)^n 
\circ Y. $$

We define the $G$-action on the domain of $Y$
to be $\gamma \otimes S \oplus \gamma \otimes T \oplus R^n$.
%
The $G$-action on $M_2(\call_B(\calg) \square M_n(A))$
is $\theta^{V \square 1}$.



%

We set
\begin{eqnarray*}
t_-(x) &=& (x \otimes 1) \oplus 0 \oplus 0^n   \square x  \\
t_+(x) &=&  V \circ ( 0  \oplus (x \otimes 1)  \oplus 0^n \big ) \circ V^* \square x 
\end{eqnarray*}

Note that 
$t_+(x) - t_-(x) \in \calk(\calg)$
by lemma \ref{lemma271}
(observe it first for a matrix $x$ with a single non-zero entry to get essentially $v_+(y)-v_-(y)\in \calk(\calf)$).

Observe that $\call(\calg) \cong M_n(\call(\calf))$
and this is the 
equivariant ``corner embedding":
%
$$\phi(T \square a)= Y^{-1} \circ (T \oplus 0^{n-1}  ) \circ Y  \square f(a)$$   

Notice that $\Phi = \phi \otimes 1_{M_2}$
is equivariant, as for example,
by observing the 
lower left corner
of the 
range of $\Phi$,
$$\phi \Big (U(H) \circ g \big(U(H)^* \circ  T \big) \square g(a) \Big )= 
V \circ g \big( V^* \circ \phi( T \square a) \big)$$
for $g \in G$ and $T \in \call(\calf)$.
%
%

It is now straightforward to check the claim with lemma \ref{lemma21}.
\end{proof}

\section{The standard form}

\label{sec14}

Recall that $G$ is a locally compact 
(not necessarily Hausdorff) groupoid or a countable inverse semigroup, and that all $C^*$-algebras are separable.

\begin{lemma}			\label{lemma191}
Let two diagrams as in (\ref{bz2})
be given, say for actions $S,T$ and homomorphisms
$s_\pm, t_\pm$. 
{\rm (i)} Then we can sum up these diagrams to 
\begin{equation}	 	
\xymatrix{
B \ar[r] &
B \otimes \calk
%
\ar[r] 
& 
\call_B \big ( (\H_B,S) \oplus (\H_B,T) \big ) \square A \ar[rr]    && A
\ar@<.5ex>[ll]^{s_- \oplus t_-, 
s_+ \oplus t_+ } 
}
\end{equation}
and this corresponds to the sum of the associated elements in $GK$, i.e.
$$s_+ \mu \Delta_{s_-} e^{-1}+
t_+ \mu \Delta_{t_-} e^{-1}
=
(s_+ \oplus t_+) \mu \Delta_{s_- \oplus t_-} e^{-1}
$$

The $M_2$-action is ${\rm Ad}(V \oplus W) \square (\alpha \otimes 1)$
for the two $M_2$-actions 
${\rm Ad}(V) \square (\alpha \otimes 1)$
and 
${\rm Ad}(W) \square (\alpha \otimes 1)$
of the given diagrams.  

(Note that we used the wrong but suggestive notation $s_- \oplus t_- := x \oplus y \square 1$ for $s_- = x \square 1, t_-=y \square 1$.)

\end{lemma}


\begin{proof}
We drop the proof. 
One does this as in 
\cite[lemma 3.6]{bremarks}.
%
\end{proof}

\begin{corollary}					\label{cor142}
Consider the diagram (\ref{bz2}).
Make its `negative' diagram
where we exchange $t_-$ and $t_+$ and
transform the $M_2$-action under coordinate flip.
Then its associated element in $GK$ is 
the negative, that is, 
$$- t_+ \mu \Delta_{t_-} e^{-1} =
t_- \mu \Delta_{t_+} e^{-1}$$

\end{corollary}

\begin{proof}


Considering a sum as in the last lemma we have
$$(t_+ \oplus t_-) \mu \Delta_{t_- \oplus t_+}
= (t_- \oplus t_+)  \Delta_{t_- \oplus t_+} =0$$
by the 
rotation homotopy
$V_s (t_- \oplus t_+) V_{-s}$,
where 
$V_s= \Big (\begin{matrix} \cos s & \sin s \\ - \sin s &  \cos s \end{matrix}
\Big ) \oplus 1 \in \call_B( \cale^2) \oplus \tilde A$
for $s \in [0,\pi/2]$
and where we define the $G$-action
on 
$M_2(X[0,1])$ 
for $X=\call_B(\cale \oplus \cale)$ by $(\theta^{V_s} 
)_{s \in [0,\pi/2]}$ 
so that we can apply lemma \ref{lemma123}.
%
%
\end{proof}

\begin{theorem}		\label{theorem212}
Every morphism in $GK$ can be represented
as $t_+ \mu \Delta_{t_-} e^{-1}$ (called standard form) for a diagram
as in (\ref{bz2}).
\end{theorem}

\begin{proof}
Write ${\rm id}_A =  {\rm id}_A e e^{-1}$ for the corner embedding $e: A \rightarrow A \otimes \calk$. 
By applying lemma \ref{lemma100} and 
lemma \ref{lemma165}.(ii) (or lemma \ref{lemma14}) 
to ${\rm id}_A e$,
we can present 
the identity homomorphism ${\rm id_A}$ in 
the claimed form.


Assume by induction on $n \ge 1$ that a word $w=w_n \ldots w_1$ of length $n$ in $GK$
allows such a presentation as claimed.
If $v$ is a homomorphism then $v w$ allows also such a presentation by lemma \ref{lemma215}.
Similarly we use 
proposition \ref{lemma4} if $v=\Delta_u$,
and lemma \ref{lemma5} if $v$ 
is the inverse of a corner embedding
to complete the induction step for $vw$.
For sum of words we 
apply lemma \ref{lemma191} 
(or \cite[lemma 3.6]{bremarks} before induction)
and corollary \ref{cor142}.
\end{proof}


\begin{corollary}			\label{cor133}
The assignment $\calb$ is multiplicative, so is a functor.
\end{corollary}



\begin{proof}
By lemma \ref{lemma106}  
we have 
$$\calb(z w)=
\calb \big (\cala(\calb(z)) \cdot \cala(\calb(w)) \big )
= 
\calb \big (\cala \big (\calb(z) \cdot \calb(w) \big ) \big)
$$
for composable morphisms $z,w \in KK^G$.
%
Since by theorem \ref{theorem212} we can
bring $\calb(z) \cdot \calb(w)$ to standard 
form,
by lemma \ref{lemma165} this is
%
%
$\calb(z) \cdot \calb(w)$.

For the unit $1_A = [0 + {\rm id}_A \oplus 0, (\H_A, R_0), 1] \in KK^G(A,A)$
we get $\calb(1_A) = {\rm id}_A$ by lemma \ref{lemma100}.
\end{proof}

\begin{corollary}			\label{corollary222}
$\calb \circ \cala = {\rm id}$.
\end{corollary}

\begin{proof}
We bring a morphism in $GK$ to standard form
by theorem \ref{theorem212}
and then apply lemma \ref{lemma165}.
\end{proof}

\begin{corollary}				\label{cor415}
The functors $\cala: GK \rightarrow KK^G$ and $\calb:KK^G \rightarrow GK$ are isomorphims of categories and inverses to each other.
In particular, $GK \cong KK^G$.
\end{corollary}

\begin{proof}
By lemma \ref{lemma165} and corollary \ref{corollary222}.
\end{proof}

\bibliographystyle{plain}
\bibliography{references}

\end{document}